\documentclass{siamltex}
\usepackage{amsmath, amsfonts}
\usepackage{amssymb,graphicx}

\newcommand{\A}{\mathbb{A}}
\newcommand{\C}{\mathbb{C}}
\newcommand{\R}{\mathbb{R}}
\newcommand{\cE}{\mathcal{E}}
\newcommand{\cF}{\mathcal{F}}
\newcommand{\cG}{\mathcal{G}}
\newcommand{\cJ}{\mathcal{J}}
\newcommand{\cK}{\mathcal{K}}
\newcommand{\cL}{\mathcal{L}}
\newcommand{\cO}{\mathcal{O}}
\newcommand{\cS}{\mathcal{S}}
\newcommand{\cV}{\mathcal{V}}
\newcommand{\Real}{\textrm{\rm Re}\,}
\newcommand{\Imag}{\textrm{\rm Im}\,}
\newcommand{\half}{\tfrac{1}{2}}
\newcommand{\iprod}[1]{\langle{#1}\rangle}
\newcommand{\dt}{{\frac{d}{dt}}}

\newtheorem{remark}[theorem]{Remark}

\title{Stability of scalar radiative shock profiles\thanks{This work was supported in part by the National 
Science Foundation award number DMS-0300487. CL, CM and RGP are warmly grateful to the 
Department of Mathematics, Indiana University, for their hospitality and financial support during two 
short visits in May 2008 and April 2009, when this research  was carried out. 
The research of RGP was partially supported by DGAPA-UNAM through the program PAPIIT, 
grant IN-109008.}}

\author{Corrado Lattanzio\footnotemark[2]
\and Corrado Mascia\footnotemark[3]
\and Toan Nguyen\footnotemark[4]
\and Ram\'on G. Plaza\footnotemark[5]
\and Kevin Zumbrun\footnotemark[4]
}

\begin{document}
\maketitle

\begin{abstract}
This work establishes nonlinear orbital asymptotic stability of scalar radiative shock 
profiles, namely,  traveling wave solutions to the simplified model system of radiating 
gas \cite{Hm}, consisting of a scalar  conservation law coupled with an elliptic equation 
for the radiation flux. 
The method is based on the derivation of pointwise Green function bounds and 
description of  the linearized  solution operator. 
A new feature in the present analysis is the construction of the resolvent kernel for the 
case of an eigenvalue system of equations of degenerate type. 
Nonlinear stability then follows in standard fashion by linear
estimates derived from these pointwise bounds, combined with
nonlinear-damping type energy estimates.
\end{abstract}

\begin{keywords}
Hyperbolic-elliptic coupled systems, Radiative shock, 
pointwise Green function bounds, Evans function.
\end{keywords}

\begin{AMS}
35B35 (34B27 35M20 76N15)
\end{AMS}

\renewcommand{\thefootnote}{\fnsymbol{footnote}}

\footnotetext[2]{Dipartimento di Matematica Pura ed Applicata,
	Universit\`a dell'Aquila, Via Vetoio, Coppito, 
	I-67010 L'Aquila (Italy)}
\footnotetext[3]{Dipartimento di Matematica ``G. Castelnuovo'',
	Sapienza, Universit\`a di Roma, P.le A. Moro 2, 
	I-00185 Roma (Italy)}
\footnotetext[4]{Department of Mathematics, Indiana University,
	Bloomington, IN 47405 (U.S.A.)}
\footnotetext[5]{Departamento de Matem\'aticas y Mec\'anica,
	IIMAS-UNAM, Apdo. Postal 20-726, C.P. 01000 M\'exico D.F. (M\'exico)}

\renewcommand{\thefootnote}{\arabic{footnote}} 

\pagestyle{myheadings} 
\thispagestyle{plain} 
\markboth{C.LATTANZIO, C.MASCIA, T.NGUYEN, R.G.PLAZA
	AND K.ZUMBRUN}{STABILITY OF SCALAR RADIATIVE SHOCK PROFILES} 

\section{Introduction}\label{sec:intro}

The one-dimensional motion of a radiating gas (due to high-temperature effects) 
can be modeled by the compressible Euler equations coupled with an elliptic equation 
for the radiative flux term \cite{Hm,VK}.
The present work considers the following simplified \textit{model system of a radiating gas}
\begin{equation}\label{modelsyst}
	\begin{aligned}
		u_t + f(u)_x + Lq_x &= 0,\\
		-q_{xx} + q + M(u)_x &= 0,
	\end{aligned}
\end{equation}
consisting of a single regularized conservation law coupled with a scalar elliptic equation. 
In \eqref{modelsyst}, $(x,t) \in \R \times [0,+\infty)$, $u$ and $q$ are scalar functions of 
$(x,t)$, $L \in \R$ is a constant, and $f, M$ are scalar functions of $u$. 
Typically, $u$ and $q$ represent velocity and heat flux of the gas, respectively. 
When the velocity flux is the Burgers flux function, $f(u) = \half u^2$, and the coupling term
$M(u) = \tilde M u$ is linear ($\tilde M$ constant), this system constitutes a good approximation 
of the physical Euler system with radiation \cite{Hm}, and it has been extensively studied by 
Kawashima and Nishibata \cite{KN1,KN2,KN3}, Serre \cite{Ser0} and Ito \cite{Ito}, among others. 
For the details of such approximation the reader may refer to \cite{KN2,KT,Hm}.

Formally, one may express $q$ in terms of $u$ as $q = - \cK M(u)_x$, where 
$\cK = (1-\partial_x^2)^{-1}$, so that system \eqref{modelsyst} represents some regularization 
of the hyperbolic (inviscid) associated conservation law for $u$. 
Thus, a fundamental assumption in the study of such systems is that
\begin{equation}\label{mainassump}
 	L\,\frac{dM}{du}(u) > 0,
\end{equation}
for all $u$ under consideration, conveying the right sign in the diffusion coming from 
Chapman--Enskog expansion (see \cite{ST}).

We are interested in traveling wave solutions to system \eqref{modelsyst} of the form
\begin{equation}\label{formTW}
	(u,q)(x,t) = (U,Q)(x-st),\qquad
	(U,Q)(\pm\infty) = (u_\pm, 0),
\end{equation}
where the triple $(u_+,u_-,s)$ is a shock front of Lax type of the underlying 
scalar conservation law for the velocity,
\begin{equation}\label{scalarcl}
	u_t + f(u)_x = 0,
\end{equation}
satisfying Rankine-Hugoniot condition $f(u_+) - f(u_-) = s(u_+ - u_-)$, 
and Lax entropy condition $\frac{df}{du}(u_+) < s < \frac{df}{du}(u_-)$. 
Morover, we assume genuine nonlinearity of the conservation law \eqref{scalarcl}, 
namely, that the velocity flux is strictly convex,
\begin{equation}\label{GNL}
 	\frac{d^2 f}{du^2}(u) > 0
\end{equation}
for all $u$ under consideration, for which the entropy condition reduces to $u_+ < u_-$. 
We refer to \textit{weak} solutions of the form \eqref{formTW} to the system \eqref{modelsyst}, 
under the Lax shock assumption for the scalar conservation law, as \textit{radiative shock profiles}. 
The existence and regularity of traveling waves of this type under hypotheses \eqref{mainassump} 
is known \cite{KN1,LMS1}, even for non-convex velocity fluxes \cite{LMS1}.

According to custom and without loss of generality, we can reduce to the case of a stationary 
profile $s=0$, by introducing a convenient change of variable and relabeling the flux function 
$f$ accordingly.
Therefore, and after substitution, we consider a stationary radiative shock profile $(U,Q)(x)$ 
solution to \eqref{modelsyst}, satisfying
\begin{equation}\label{eq:profileeqn}
	\begin{aligned}
		f(U)' + L\,Q' &= 0,\\
		-Q'' + Q + M(U)' &= 0,
	\end{aligned}
\end{equation} 
(here ${}^\prime$ denotes differentiation with respect to $x$),
connecting endpoints $(u_\pm,0)$ at $\pm\infty$, that is,
\begin{equation*}
	\lim_{x\to \pm\infty} (U,Q)(x) = (u_\pm,0).
\end{equation*}
Therefore, we summarize our main structural assumptions as follows:
\begin{align}
	\label{A0} \tag{A0}
		f, M \in C^5, \qquad &\text{(regularity),}\\
	\label{A1} \tag{A1}
		 \frac{d^2 f}{du^2}(u) > 0, \qquad &\text{(genuine nonlinearity),}\\
	\label{A2}\tag{A2}
		f(u_-) = f(u_+), \qquad &\text{(Rankine-Hugoniot condition),} \\
	\label{A3}\tag{A3}
		u_+ < u_-, \qquad &\text{(Lax entropy condition),}\\
	\label{A4} \tag{A4}
		L\, \frac{dM}{du}(u)> 0, \qquad &\text{(positive diffusion)},
\end{align}
where $u \in [u_+,u_-]$. 
For concreteness let us denote
\begin{equation}\label{defofag}
a(x) :=  \frac{df}{du}(U(x)), \quad b(x) := \frac{dM}{du}(U(x)),
\end{equation}
and assume (up to translation) that $a(0) = 0$. 
Besides the previous structural assumptions we further suppose that
\begin{equation}\label{A5k}\tag{A5$_k$}
	Lb(0) + (k + \tfrac{1}{2})a'(0) > 0, \qquad k =1,\ldots, 4.
\end{equation}

\begin{remark}\label{remarkaA4}\rm
Under assumption (A4), the radiative shock profile is monotone, and, as shown later
on, the spectral stability condition holds.
Let us stress that, within the analysis of the linearized problem and of the nonlinear 
stability, we only need (A4) to hold at the end states $u_\pm$ and at the degenerating 
value $U(0)$.
 \end{remark}

\begin{remark}\label{remarka5k}\rm 
Hypotheses (A5$_k$) are a set of additional technical assumptions inherited from the 
present stability analysis (see the establishment of $H^k$ energy estimates of  Section
 \ref{sec:damping} below, and of pointwise reduction bounds in Lemma \ref{lem-pttracking})
and are not necessarily sharp.
It is worth mentioning, however, that assumptions \eqref{A5k}, with $k=1, \ldots, 4$, are satisfied, 
for instance, for all profiles with small-amplitude $|u_- - u_+|$, in view of \eqref{mainassump} 
and $|U'| = \mathcal{O}(|u_- - u_+|)$.
\end{remark}

In the present paper, we establish the asymptotic stability of the
shock profile $(U,Q)$ under small initial perturbation. 
Nonlinear wave behavior for system \eqref{modelsyst} and its generalizations has been the 
subject of thorough research over the last decade. 
The well-posedness theory is the object of study in \cite{LM1,KaNN1,KaNN2, IK} and \cite{diF07}, 
both for the simplified model system and more general cases.
The stability of constant states \cite{Ser0}, rarefaction waves \cite{KT, GRZ}, 
asymptotic profiles \cite{Lau05, DL, DFL}  for the model system with Burgers flux has been 
addressed in the literature.

Regarding the asymptotic stability  of radiative shock profiles, the problem has been previously 
studied by Kawashima and  Nishibata \cite{KN1} in the particular case of Burgers velocity flux 
and for linear $M = \tilde M u$,  which is one of the few available stability results for \textit{scalar} 
radiative shocks in the literature\footnote{The other scalar result is the partial analysis of 
Serre \cite{Ser7} for the exact  Rosenau model; in the case of systems, we mention the stability 
result of \cite{LCG2} for the full  Euler radiating system under zero-mass perturbations, based on 
an adaptation of the classical  energy method of Goodman-Matsumura-Nishihara \cite{Go1,MN}.}. 
In \cite{KN1}, the authors establish asymptotic stability with basically the same rate of decay 
in $L^2$ and under fairly similar assumptions as we have here. 
Their method, though, relies on integrated coordinates and $L^1$ contraction property, 
a technique which may not work for  the system case (i.e., $\tilde u\in\R^n$, $n\ge2$).
In contrast, we provide techniques which may be extrapolated to systems, enable us to handle 
variable $\frac{dM}{du}(u)$, and provide a large-amplitude theory based on spectral stability 
assumptions in cases that linearized stability is not automatic (e.g., system case, or 
$\frac{dM}{du}(u)$ variable). 
These technical considerations are some of the main motivations for the present analysis.

The nonlinear asymptotic stability of traveling wave solutions to models in continuum mechanics, 
more specifically, of shock profiles under suitable regularizations of hyperbolic systems of conservation 
laws, has been the subject of intense research in recent years 
(see, e.g., \cite{Ho,ZH,MaZ1,MaZ3,MaZ4,Z3,Z4,Z7,RZ,NZ2,KZ}). 
The unifying methodological approach of these works consists of refined semigroup techniques 
and the establishment of sharp pointwise bounds on the Green function associated to the linearized 
operator around the wave, under the assumption of spectral stability. 
A key step in the analysis is the construction of the resolvent kernel, together with 
appropriate spectral bounds. 
The pointwise bounds on the Green function follow by the inverse Laplace transform 
(spectral resolution)  formula \cite{ZH,MaZ3,Z3}. 
The main novelty in the present case is the extension of the method to a situation in which 
the eigenvalue  equations are written as a degenerate first order ODE system 
(see discussion in Section \ref{subsec:framework} below). 
Such extension, we hope, may serve as a blueprint to treat other model systems for which 
the  resolvent equation becomes singular. 
This feature is also one of the main technical contributions of the present analysis.

\subsection{Main results}

In the spirit of \cite{ZH,MaZ1,MaZ4,MaZ5}, we first consider solutions to \eqref{modelsyst} 
of the form $(u+U,q+Q)$, being now $u$ and $q$ perturbations, and study the linearized 
equations of \eqref{modelsyst} about the profile $(U,Q)$, which read,
 \begin{equation}
    \begin{aligned}\begin{matrix}
    u_{t}+ (a(x) u)_{x} +L q_{x}=0, \\
    - q_{xx} +q +(b(x) u)_{x} =0,\end{matrix}
    \end{aligned}
     \label{eq:lin}
\end{equation}
with initial data $u(0)=u_0$ (functions $a, b$ are defined in \eqref{defofag}). 
Hence, the Laplace transform applied to system \eqref{eq:lin} gives
\begin{equation}
    \begin{aligned}
    \lambda u+ (a(x) u)' +L q'&= S, \\
    - q'' +q +(b(x) u)' &=0,
    \end{aligned}
     \label{eq:linLaplace}
\end{equation}
where source $S$ is the initial data $u_0$. 

As it is customary in related nonlinear wave stability analyses 
\cite{AGJ,San,ZH,GZ,MaZ1,MaZ3,Z3,Z7},  we start by studying the underlying spectral 
problem, namely, the homogeneous version  of system \eqref{eq:linLaplace}:
\begin{equation}
\label{spectralsyst}
	\begin{aligned}
		(a(x) u)' &= -\lambda\,u - Lq',\\
		q'' &= q + (b(x) u)'.
	\end{aligned}
\end{equation}
An evident necessary condition for orbital stability is the absence of $L^2$ solutions 
to \eqref{spectralsyst} for values of $\lambda$ in $\{ \Real \lambda \geq 0\} \backslash \{0\}$, 
being $\lambda = 0$ the eigenvalue associated to translation invariance. 
This strong spectral stability can be expressed in terms of the \textit{Evans function}, 
an analytic function playing a role for differential operators analogous to that played by the 
characteristic polynomial for finite-dimensional operators (see \cite{AGJ,San,GZ,ZH,MaZ1,MaZ3,Z4,Z3,Z7} 
and the references therein). The main property of the Evans function is that, on the resolvent set 
of a certain operator ${\mathcal L}$, its zeroes coincide in both location and multiplicity with 
the eigenvalues of ${\mathcal L}$.

In the present case and due to the degenerate nature of system \eqref{spectralsyst} 
(observe that $a(x)$ vanishes at $x=0$) the number of decaying modes at $\pm \infty$, spanning 
possible eigenfunctions, depends on the region of space around the singularity 
(see Section \ref{sec:resolker} below, Remark \ref{rem:notconsistent}). 
Therefore, we define the following {\it stability criterion}, where the analytic functions 
$D_\pm(\lambda)$ (see their definition in \eqref{Evansfns-def} below) denote the two Evans 
functions associated with the linearized operator about the profile in regions $x\gtrless 0$, 
correspondingly,  analytic functions whose zeroes away from the essential spectrum agree 
in location and multiplicity  with the eigenvalues of the linearized operator or 
solutions of \eqref{spectralsyst}:
\begin{equation}\label{Dcond}\tag{D}
	\text{There exist no zeroes of} \;\; D_\pm(\cdot) \; 
	\text{in the non-stable half plane} \;\; \{\Real \lambda \ge 0\} \setminus \{0\}.
\end{equation}
Our main result is then as follows.
\medskip

\begin{theorem}\label{theo-main} 
Assuming \eqref{A0}--\eqref{A5k}, and the spectral stability condition \eqref{Dcond}, 
then the Lax radiative shock profile $(U,Q)$ is asymptotically orbitally stable. 
More precisely, the solution $(\tilde u,\tilde q)$ of \eqref{modelsyst} with initial data
$\tilde u_0$ satisfies
\begin{equation*}
	\begin{aligned}
		&|\tilde u(x,t) - U(x-\alpha(t))|_{L^p} 
			\le C(1+t)^{-\frac 12(1-1/p)}|u_0|_{L^1\cap H^4}\\
		&|\tilde u(x,t) - U(x-\alpha(t))|_{H^4} 
			\le C(1+t)^{-1/4}|u_0|_{L^1\cap H^4}
\end{aligned}
\end{equation*}
and
\begin{equation*}
	\begin{aligned}
		&|\tilde q(x,t) - Q(x-\alpha(t))|_{W^{1,p}} 
			\le C(1+t)^{-\frac12(1-1/p)}|u_0|_{L^1\cap H^4}\\
		&|\tilde q(x,t) - Q(x-\alpha(t))|_{H^5} \le C(1+t)^{-1/4}|u_0|_{L^1\cap H^4}
\end{aligned}
\end{equation*}
for initial perturbation $u_0:=\tilde u_0 - U$ that are sufficiently small in $L^1\cap H^4$, 
for all $p\ge 2$, for some $\alpha(t)$ satisfying $\alpha(0)=0$ and
\begin{equation*}
	|\alpha(t)|\le C|u_0|_{L^1\cap H^4},\qquad
 	|\dot\alpha(t)|\le C(1+t)^{-1/2}|u_0|_{L^1\cap H^4},
\end{equation*}
where $\dot{}$ denotes the derivative with respect to $t$.
\end{theorem}
\medskip

\begin{remark}\rm
The time-decay rate of $q$ is not optimal. In fact, it can be improved
as we observe that $|q(t)|_{L^2} \le C|u_x(t)|_{L^2}$ and $|u_x(t)|_{L^2}$ is 
expected to decay like $t^{-1/2}$; we omit, however, the details of the proof.
\end{remark}

The second result of this paper is the verification of the spectral stability condition 
\eqref{Dcond} under particular circumstances.
\medskip

\begin{proposition}\label{prop-verifyD} 
The spectral stability condition \eqref{Dcond} holds under either \\
(i) $b$ is a constant; or,\\
(ii) $|u_+-u_-|$ is sufficiently small.
\end{proposition}

\begin{proof} See Appendix \ref{appx-condD}.\end{proof}

\subsection{Discussions}
Combining Theorem \ref{theo-main} and Proposition \ref{prop-verifyD}, we partially recover 
the results of \cite{KN1} for the Burgers flux and constant $L$, $M$, and at the same time 
extend them to general convex flux and quasilinear $M$.
We note that the stability result of \cite{KN1} was for all smooth shock profiles, for which 
the boundary (see \cite{KN1}, Thm. 1.25(ii)(a)) is the condition $LM=1=-a'(0)$; that is, their 
results hold whenever $LM+a'(0)>0$.
By comparison, our results hold on the smaller set of waves for which $LM+(9/2)a'(0)>0$; 
see Remark \ref{remarka5k}.
By estimating high-frequency contributions explicitly, rather than by the simple energy 
estimates used here, we could at  the expense of further effort reduce these conditions to 
the single condition 
\begin{equation}\label{single}
	LM+2a'(0)>0
\end{equation}
used to prove Lemma \ref{lem-pttracking}.
Elsewhere in the analysis, we need only $LM+a'(0)>0$; however, at the moment we 
do not see how to remove \eqref{single} to recover the full result of \cite{KN1} in the 
special case considered there.
The interest of our technique, rather, is in its generality ---particularly the possibility to 
extend to the system case--- and in the additional information afforded by the pointwise 
description of behavior, which seems interesting in its own right.

\subsection{Abstract framework}\label{subsec:framework}
Before beginning the analysis, we orient ourselves with a few simple
observations framing the problem in a more standard way. 
Consider now the inhomogeneous version
 \begin{equation}\label{eq:fulllin}
    \begin{aligned}
    u_{t}+ (a(x) u)_{x} +L q_{x}&=\varphi,\\
    - q_{xx} +q +(b(x) u)_{x} &=\psi,
    \end{aligned}   
\end{equation}
of \eqref{eq:lin}, with initial data $u(x,0)=u_0$. 
Defining the compact operator $\cK:=(-\partial_x^2+ 1)^{-1}$ 
and the bounded operator
\begin{equation*}
 	\cJ\,u:= -L\,\partial_x\,\cK\,\partial_x (b(x) u),
\end{equation*}
we may rewrite this as a nonlocal equation
 \begin{equation}\label{eq:redlin}
	\begin{aligned}
    		u_{t}+ (a(x) u)_{x} + \cJ u&=\varphi-L\,\partial_x\left(\cK\,\psi\right),\\
		u(x,0)&=u_0(x)
	\end{aligned}
\end{equation}
in $u$ alone.
The generator $\cL:= - (a(x) u)_{x} - \cJ u$ of \eqref{eq:redlin} is a zero-order 
perturbation of the generator $- a(x) u_x$ of a hyperbolic equation, so 
it generates a $C^0$ semigroup $e^{\cL t}$ and an associated Green 
distribution $G(x,t;y):=e^{\cL t}\delta_y(x)$.
Moreover, $e^{\cL t}$ and $G$ may be expressed through the inverse
Laplace transform formulae
\begin{equation}\label{iLT}
	\begin{aligned}
		e^{\cL t}&=\frac{1}{2\pi i} \int_{\gamma -i\infty}^{\gamma+i\infty}
		e^{\lambda t} (\lambda-\cL)^{-1}d\lambda,\\
		G(x,t;y)&=\frac{1}{2\pi i} \int_{\gamma -i\infty}^{\gamma+i\infty}
		e^{\lambda t} G_\lambda(x,y)d\lambda,\\
	\end{aligned}
\end{equation}
for all $\gamma\ge \gamma_0$ (for some $\gamma_0>0$), 
where $G_\lambda(x,y):=(\lambda-\cL)^{-1}\delta_y(x)$  is the resolvent kernel of $\cL$.

Collecting information, we may write the solution of \eqref{eq:fulllin} using 
Duhamel's principle/variation of constants as
\begin{equation*}
	\begin{aligned}
		u(x,t)&= \int_{-\infty}^{+\infty} G(x,t;y)u_0(y)dy \\
		&\quad + \int_0^t \int_{-\infty}^{+\infty} G(x,t-s;y)
		(\varphi-L\,\partial_x\left(\cK\,\psi\right))(y,s)\, dy\, ds,\\
		q(x,t)& = \cK\bigl( \psi-\partial_x\left(b(x) u \right)\bigr) (x,t),
	\end{aligned}
\end{equation*}
where $G$ is determined through \eqref{iLT}.

That is, the solution of the linearized problem reduces to finding the Green kernel 
for the $u$-equation alone, which in turn amounts to solving the resolvent equation 
for $\cL$ with delta-function data, or, equivalently, solving the differential equation
\eqref{eq:linLaplace} with source $S=\delta_y(x)$. 
We shall do this in standard fashion by writing \eqref{eq:linLaplace} as a first-order
system and solving appropriate jump conditions at $y$ obtained by the requirement 
that $G_\lambda$ be a distributional solution of the resolvent equations.

This procedure is greatly complicated by the circumstance that the resulting $3\times 3$ 
first-order system, given by
\begin{equation*}
	\left(\Theta(x) W\right)_x = \A(x,\lambda)W \qquad\textrm{where}\quad
	\Theta(x):= \begin{pmatrix} a(x) & 0 \\ 0  & I_2 \end{pmatrix},
\end{equation*}
is  {\it singular}  at the special point where $a(x)$ vanishes.
However, in the end we find as usual that $G_\lambda$ is uniquely determined by these criteria, 
not only for the values $\Real \lambda \ge \gamma_0>0$ guaranteed by $C^0$-semigroup theory/energy 
estimates, but, as in the usual nonsingular case \cite{He}, on the {\it set of consistent splitting}
for the first-order system, which includes all of $\{\Real \lambda \ge  0\}\setminus \{0\}$. 
This has the implication that the essential spectrum of $\cL$ is  confined to 
$\{\Real \lambda<0\}\cup \{0\}$.

\begin{remark}\label{jump}\rm
The fact (obtained by energy-based resolvent estimates) that $\cL-\lambda$ is coercive for 
$\Real \lambda \ge \gamma_0$ shows by elliptic theory that the resolvent is well-defined and unique in
class of distributions for $\Real \lambda $ large, and thus the resolvent kernel may be determined 
by the usual construction using appropriate jump conditions. 
That is, from standard considerations, we see that the construction {\it must}  work, despite the 
apparent wrong dimensions of decaying manifolds  (which happen for any $\Real \lambda >0$).
\end{remark}

To deal with the singularity of the first-order system is the most delicate and novel part of the 
present analysis. 
It is our hope that the methods we use here may be of use also in other situations where
the resolvent equation becomes singular, for example in the closely related situation of 
relaxation systems discussed in \cite{MaZ1,MaZ5}.

\subsection*{Plan of the paper}

This work is structured as follows. 
Section \ref{sec:prel} collects some of the properties of radiative profiles, and contains a technical 
result which allows us to rigorously define the resolvent kernel near the singularity. 
The central Section \ref{sec:resolker} is devoted to the construction of the resolvent kernel, based 
on the analysis of solutions to the eigenvalue equations both near and away from the singularity. 
Section \ref{sec:reskbounds} establishes the crucial low frequency bounds for the resolvent kernel. 
The following Section \ref{sec:ptwise} contains the desired pointwise bounds for the 
``low-frequency'' Green function, based on the spectral resolution formulae. 
Section \ref{sec:damping} establishes an auxiliary nonlinear damping energy estimate. 
Section \ref{sec:hfest} deals with the high-frequency region by establishing energy estimates on 
the solution operator directly. 
The final Section \ref{sec:nonlinear} blends all previous estimations 
into the proof of the main nonlinear stability result (Theorem \eqref{theo-main}). 
We also include three Appendices, which contain, a pointwise extension of the Tracking lemma, 
the proof of spectral stability under linear coupling or small-amplitude assumptions, and the 
monotonicity of general scalar profiles, respectively.

\section{Preliminaries}\label{sec:prel}

\subsection{Structure of profiles}

Under definition \eqref{defofag}, we may assume (thanks to translation invariance; 
see Remark \ref{regulrem} below) that $a(x)$ vanishes exactly at one single point 
which we take as $x =0$. 
Likewise, we know that the velocity profile is monotone decreasing (see \cite{LMS1,LMS2,Ser7} 
or Lemma \ref{monoU} below), that is $U'(x)<0$, which implies, 
in view of genuine nonlinearity \eqref{GNL}, that
\begin{equation*}
	a'(x) < 0 \qquad \forall\,x\in\R
	\qquad\textrm{and}\qquad
	 x\,a(x) < 0\qquad\forall\,x\neq 0.
\end{equation*}
From the profile equations we obtain, after integration, that
\begin{equation*}
	LQ = f(u_\pm) - f(U) > 0,
\end{equation*}
for all $x$, due to Lax condition. 
Therefore, substitution of the profile equations \eqref{eq:profileeqn} yields the relation
\begin{equation*}
	\bigl(a'(x) + L\, b(x)\bigr)U' = -LQ - a(x)U'',
\end{equation*}
which, evaluating at $x=0$ and from monotonicity of the profile, implies that
\begin{equation}\label{eq:profdiff}
	a'(0) + L\,b(0) > 0.
\end{equation}
Therefore, the last condition is a consequence of the existence result 
(see Theorem \ref{existencethm} below), and it will be used throughout. 
Notice that (A$5_1$) implies condition \eqref{eq:profdiff}.

Next, we show that the waves decay exponentially to their end states, a crucial fact 
in the forthcoming analysis.
\medskip

\begin{lemma}\label{lem-expdecay} 
Assuming \eqref{A0} - \eqref{A4},  a radiative shock profile $(U,Q)$ of
\eqref{modelsyst} satisfies
\begin{equation}\label{layerdecay}
	\Big|(d/dx)^k(U-u_\pm,Q)\Big|\le C e^{-\eta |x|},\quad k=0,...,4,
\end{equation}
as $|x|\to +\infty$, for some $\eta>0$.
\end{lemma}
\medskip

\begin{proof}
As $|x|\to +\infty$, defining $a_\pm = a(\pm\infty)$ and  $b_\pm = b(\pm\infty)$, we 
consider the asymptotic system of \eqref{eq:profileeqn}, that is the constant
coefficient linear system 
\begin{equation*}
    \begin{aligned}
        a_\pm U' &= - L Q',\\
       -Q'' + Q &= - b_\pm U',
    \end{aligned}
\end{equation*}
which, by substituting $U'$ into the second equation, becomes
\begin{equation*}
    -Q'' - \frac{ Lb_\pm}{a_\pm}\, Q' + Q =0,
\end{equation*}
or equivalently,
\begin{equation*}
	\begin{pmatrix}Q\\Q'\end{pmatrix}'  
	=A_Q\begin{pmatrix}Q\\Q'\end{pmatrix}, \qquad 
	\textrm{\rm with  }
	A_Q:=\begin{pmatrix}0&1\\1&-Lb_\pm/a_\pm \end{pmatrix},
\end{equation*}
which then gives the exponential decay estimate \eqref{layerdecay}
for $Q$ by the hyperbolicity of the matrix $A_Q$, that is,
eigenvalues of $A_Q$ are distinct and nonzero. Estimates for $U$
follow immediately from those for $Q$ and the relation
\begin{equation*}
    LQ = f(u_{\pm}) - f(U),
\end{equation*}
obtained by integrating the first equation of \eqref{eq:profileeqn}.
\end{proof}

\subsection{Regularity of solutions near $x = 0$}

In this section we establish some analytic properties of the solutions to system 
\eqref{spectralsyst} near the singularity, which will be used during the construction 
of the resolvent kernel in the central Section \ref{sec:resolker} below. 
Introducing the variable $p:=b(x) u - q'$, system \eqref{spectralsyst} takes the 
form of a first-order system, which reads
\begin{equation}\label{eq:linLaplace2}
    \begin{aligned}
    	a(x)u' &= - \left ( \lambda  + a'(x) + L b(x) \right )u +L p,\\
    	q' &= b(x) u - p,\\
    	p' &= - q.
    \end{aligned}   
\end{equation}
For technical reasons which will be clear from the forthcoming analysis, in order to define 
the transmission conditions in the definition of the resolvent kernel (which is defined as 
solutions to the conservative form of system \eqref{eq:linLaplace2} in distributional sense
with appropriate jump conditions; see Section \ref{sec:outline} below), we need $p$ and 
$q$ to be regular across the singularity $x=0$ (having finite limits at both sides), $u$ to have 
(at most) an integrable singularity at that point, namely, that $u \in L^{1}_{loc}$ near zero 
(away from zero it is bounded, so this is trivially true), and that it verifies $a(x)u\to 0$ as $x\to 0$. 
These properties are proved in the next technical lemma.
\medskip

\begin{lemma}\label{lem:tech}
Given $\lambda\in\C$, set $\nu:={\Real \lambda +a'(0) +Lb(0)}/{|a'(0)|}$.
Under assumptions \eqref{A0}-\eqref{A4}, and $\Real \lambda >-Lb(0)$, 
then any solution of \eqref{eq:linLaplace2} verifies
\begin{enumerate}
	\item[1.]  $\displaystyle{ |u(x)|  \leq C\,|x|^{\nu} }$ for $x\sim 0$ and for some $C>0$;
	\item[2.]  $q$ is absolutely continuous and $p$ is $C^{1}$ (for $x\sim 0$),
\end{enumerate}
In particular, $u\in L^{1}_{loc}$ (for $x\sim 0$) and $a(x)u(x)\to 0$ as $x\to 0$.
\end{lemma}
\medskip

The proof will be done in two steps: (i) first, taking into
account ``elliptic regularity'' in the equation for $p$,
     \begin{equation}
         -p''+p = b(x) u,
         \label{eq:forp}
     \end{equation}
we prove the $L^{1}_{loc}$ bound for $u$ close to zero and the
subsequent regularity for $p$ and $q$; and (ii), using such a bound,
we then prove the pointwise control given in $(i)$.

Alternatively, one can explicitly solve the above elliptic equation for $p$ and get directly 
the pointwise result for $u$ by plugging the relation into the Duhamel formula for $u$.
Finally, such a control gives the $L^{1}_{loc}$ property for $u$ and all other regularity properties.

\begin{proof}[Proof of Lemma \ref{lem:tech}]
Let us consider the case $x\geq 0$, the case $x\leq 0$ being similar. 
Consider a fixed $x_{0}>0$, to be chosen afterwards and let $(u,q,p)$ be any solution 
of (\ref{eq:linLaplace2}) emanating from that point. 
Therefore, from \eqref{eq:forp} we know that
\begin{equation}\label{eq:pformula}
	p(x) = C_{1}e^{-x} + C_{2}e^{x} + \int_{x_{0}}^{x} g(x,y)u(y)dy
\end{equation}
for a given (regular) kernel $g(x,y)$.
Therefore there exists a constant $C_{x_{0}}$ such that for any $\epsilon>0$
\begin{equation*}
	|p|_{L^{\infty}(\epsilon,  x_{0})} \leq C_{x_{0}} (1 + |u|_{L^{1}(\epsilon, x_{0})}).
\end{equation*}
Note that the constant $C_{x_{0}}$ is uniform on $\epsilon$ and it stays bounded as 
$x_{0}$ approaches zero and it depends only on the initial values $p(x_{0})$, 
$q(x_{0}) = - p'(x_{0})$.
Moreover, the Duhamel principle gives for any $x\in [\epsilon,  x_{0}]$:
\begin{align}\label{eq:duhamelu}
	u(x) 	&= u(x_{0}) \exp \left (- \int_{x_{0}}^{x} \frac{\lambda + a'(y)
		+ L\,b(y)}{a(y)}dy \right ) \nonumber\\
	& \ + L \int_{x_{0}}^{x}\frac{1}{a(y)} \exp \left (- \int_{y}^{x} \frac{\lambda + a'(z)
		+ L\,b(z)}{a(z)} dz\right )p(y)dy.
    \end{align}
From \eqref{GNL} we obtain 
\begin{equation*}
      \frac{\lambda + a'(x)+ L\,b(x)}{a(x)} \sim 
      \frac{\lambda + a'(0) + L\,b(0)}{a'(0)x},\ \hbox{for}\ x\sim 0.
\end{equation*}
Hence, for $x\sim 0$,
\begin{equation}\label{eq:integrability}
	\begin{aligned}
      	\exp \left (- \int_{x_0}^{x} \frac{\lambda + a'(y)+L\,b(y)}{a(y)}dy \right ) 
      	& \sim \exp \left (- \int_{x_0}^{x} \frac{\lambda + a'(0)
              + L\,b(0)}{a'(0)y}dy \right )\\
    	& = \left | \frac{x}{x_0} \right |^{-\frac{\lambda + a'(0)
              + L\,b(0)}{a'(0)}}.            
	\end{aligned}
\end{equation}
Hence the first term of \eqref{eq:duhamelu} is integrable in $[0,x_{0}]$ provided 
$\Real \lambda >-L\,b(0)$, being $a'(0)<0$ (our argument applies for 
$\lambda \neq -L\,b(0) - a'(0)$; for $\lambda = -L\,b(0) - a'(0)$ all functions in the 
integrals above are indeed bounded at zero and the proof of the lemma is even simpler).
Thus, for a constant $C_{x_{0}}$ as above,
 \begin{align}\label{eq:l1u1}
	|u|_{L^{1}(\epsilon,  x_{0})} &\leq u(x_{0})C_{x_{0}}
     		+ C_{x_{0}} (1 + |u|_{L^{1}(\epsilon,  x_{0})})\times\nonumber\\
    	& \ \times\int_{\epsilon}^{x_{0}}\int_{x_{0}}^{x}\frac{1}{|a(y)|}
      		\exp \left (- \int_{y}^{x} \frac{\Real \lambda + a'(z)
          	+ L\,b(z)}{a(z)} dz\right )dy\,dx.
\end{align}
Now we use again \eqref{eq:integrability} to estimate the
integral term in \eqref{eq:l1u1} as follows:
\begin{align*}
	 \int_{x_{0}}^{x}\frac{1}{|a(y)|}\,
          	&\exp \left (- \int_{y}^{x} \frac{\Real \lambda + a'(z)
            	+ L\,b(z)}{a(z)} dz\right )dy\\   
	&\sim \int_{x_{0}}^{x}\frac{1}{|a'(0)y|}
                	\left | \frac{x}{y} \right |^{\nu} dy
		= \frac{|a'(0)|\,x^{\nu}}{\Real \lambda + a'(0) + L\,b(0)} 
           	\left (x^{-\nu} -  x_{0}^{-\nu}\right ) \\
  	& = - \frac{1}{\Real \lambda + a'(0) + L\,b(0)} 
		\left (1 -  \left (\frac{x}{x_{0}}\right )^{\nu}\right ).
\end{align*}
Therefore,
\begin{align*}
       |u|_{L^{1}(\epsilon,  x_{0})} &\leq u(x_{0})C_{x_{0}}
           +  C_{x_{0}} (1 +
          |u|_{L^{1}(\epsilon,  x_{0})})x_{0}\nonumber\\
        & \ + C_{x_{0}} (1 +
          |u|_{L^{1}(\epsilon,  x_{0})})x_{0}^{-\nu}\frac{1}{\Real \lambda +L\,b(0)}
          x_{0}^{-\frac{\Real \lambda
           + L\,b(0)}{a'(0)}} \\
        & = u(x_{0})C_{x_{0}}
           +  C_{x_{0}} (1 +
          |u|_{L^{1}(\epsilon,  x_{0})})x_{0}.
   \end{align*}
   Finally, for a sufficiently small, but \emph{fixed} $x_{0}>0$,
   from the above relation we conclude
   \begin{equation*}
       |u(x)|_{L^{1}(\epsilon,  x_{0})} \leq C_{x_{0}}
   \end{equation*}
   uniformly in $\epsilon$, namely, $u\in L^{1}(0,\epsilon_0)$ for $\epsilon_0>0$.
   At this point, part 2.\ of the lemma is an easy consequence of
   expressions \eqref{eq:pformula}, \eqref{eq:linLaplace2}$_{2}$ and
   \eqref{eq:linLaplace2}$_{3}$.

Once we have obtained the $L^{1}_{loc}$ property of $u$ at zero, we
know in particular $|p|_{L^{\infty}(0, x_{0})}$ is bounded. 
Hence we can repeat all estimates on the integral terms of
\eqref{eq:duhamelu} to obtain part 1.\ of the lemma. 
Finally,
\begin{equation*}
        \lim_{x\to 0}a(x)u(x) =0
\end{equation*}
is again a consequence of $\Real \lambda >-Lb(0)$.
\end{proof}

\begin{remark}\label{rem:udecay}\rm
From condition (\ref{eq:profdiff}) it is clear that, for $\Real \lambda<0$, but sufficiently
close to zero, $u(x)$ is not blowing up for $x\to 0$, but it vanishes in that limit,
regardless of the shock strength (the negative term $a'(0)$ approaches zero as the 
strength of the shock tends to zero).
\end{remark}

\section{Construction of the resolvent kernel}\label{sec:resolker}

\subsection{Outline}
\label{sec:outline}

Let us now construct the resolvent kernel for  $\cL$, or equivalently, the solution 
of \eqref{eq:linLaplace2} with delta-function source in the $u$ component. 
The novelty in the present case is the extension of this standard method
to a situation in which the spectral problem can only be written as
a {\it degenerate} first order ODE. 
Unlike the real viscosity and relaxation cases \cite{MaZ1,MaZ3,MaZ4,MaZ5} 
(where the operator ${\mathcal L}$, although degenerate, yields a non-degenerate 
first order ODE in an appropriate reduced space), here we deal with the resolvent
system  for the unknown $W := (u,q,p)^\top$ 
\begin{equation}\label{eqW}
	\left(\Theta(x) W\right)' = \A(x,\lambda)W, 
\end{equation}	
where
\begin{equation*}
	\Theta(x):= \begin{pmatrix} a(x) & 0 \\ 0  & I_2 \end{pmatrix},\qquad
	\A(x):= \begin{pmatrix} 
	 -(\lambda+L\,b(x)) & 0 & L \\ b(x) & 0 & -1 \\ 0  & -1 & 0 
	 	\end{pmatrix},
\end{equation*}
that degenerates at $x=0$.

To construct the resolvent kernel $\cG_\lambda=\cG_\lambda(x,y)$, we solve
\begin{equation}\label{eq:resolker}
	\partial_x\left(\Theta(x)\,\cG_\lambda\right) - \A(x,\lambda)\,\cG_\lambda= \delta_y(x)\,I,
\end{equation}
in the distributional sense, so that
\begin{equation*}
	\partial_x\left(\Theta(x)\,\cG_\lambda\right) - \A(x,\lambda)\,\cG_\lambda = 0,
\end{equation*}
for all $x\ne y$ with appropriate jump conditions (to be determined) at $x=y$. 
The first element in the first row of the matrix-valued function $\cG_\lambda$ is the 
resolvent kernel $G_\lambda$ of $\cL$ that we seek.

\subsection{Asymptotic behavior}

First, we study the asymptotic behavior of solutions to the spectral system
\begin{equation}
\label{specsyst}
	\begin{aligned}
		a(x)u' &= - (\lambda + a'(x) + L\,b(x)) u + Lp,\\
		q' &=b(x) u - p,\\
		p' &= -q,
	\end{aligned}
\end{equation}
away from the singularity at $x=0$, and for values of $\lambda \neq 0$,
$\Real \lambda \geq 0$. 
We pay special attention to the small frequency regime, $\lambda \sim 0$. 
Denote the limits of the coefficients as
\begin{equation*}
	a_\pm := \lim_{x\to\pm\infty} a(x) =\frac{df}{du}(u_\pm), \qquad
	b_\pm := \lim_{x\to\pm\infty} b(x) =\frac{dM}{du}(u_\pm).
\end{equation*}
From the structure of the wave we already have that $a_+ < 0 < a_-$.
The asymptotic system can be written as
\begin{equation}
\label{asympsyst}
 	W' = \A_\pm(\lambda)W,
\end{equation}
where  
\begin{equation*}
	\A_\pm(\lambda) := \begin{pmatrix} -a_\pm^{-1}(\lambda + Lb_\pm) & 0
	& a_\pm^{-1}L \\ b_\pm & 0 & -1 \\ 0 & -1 & 0 \end{pmatrix}.
\end{equation*}
To determine the dimensions of the stable/unstable eigenspaces, 
let $\lambda \in \R^+$, $\lambda \to +\infty$. 
The characteristic polynomial reads
\begin{equation*}
	\pi_\pm(\mu) := |\mu\,I - \A_\pm(\lambda)| 
	=\mu^3 + a_\pm^{-1}(\lambda + Lb_\pm) \mu^2 - \mu - a_\pm^{-1} \lambda,
\end{equation*}
for which
\begin{equation*}
	\frac{d\pi_\pm}{d\mu} = 3\mu^2 + 2a_\pm^{-1}(\lambda + Lb_\pm) \mu -1,
\end{equation*}
has one negative and one positive zero, regardless of the sign of
$a_\pm$, for each $\lambda \gg 1$; they are local extrema of $\pi_\pm$. 
Since $\pi_\pm \to \pm\infty$ as $\mu \to \pm\infty$, $\pi_\pm(0) = -a_\pm^{-1} \lambda$ 
has the opposite sign  with respect to $a_\pm$ and
\begin{equation*}
 \pi_\pm(-a_\pm\,\lambda)=a_\pm\left(a_\pm^2+\frac{1}{a_\pm^4}\right)\,\lambda^3
 +o(\lambda^3)\qquad\lambda\to\infty,
\end{equation*}
so that $\pi_-/\pi_+$ is positive/negative at some negative/positive value of $\mu$,
we get two positive and one negative zeroes for 
$\pi_+$, and two negative and one positive zeroes for $\pi_-$, whenever 
$\lambda \in \R^+$, $\lambda \gg 1$.

We readily conclude that for each $\Real \lambda > 0$, there exist two
unstable eigenvalues $\mu_1^+(\lambda)$ and $\mu_2^+(\lambda)$ with
$\Real \mu > 0$, and one stable eigenvalue $\mu_3^+(\lambda)$ with
$\Real \mu < 0$. The stable $S^+(\lambda)$ and unstable $U^+(\lambda)$
manifolds (solutions which decay, respectively, grow at $+\infty$)
have, thus, dimensions
\begin{equation*}
	\dim U^+(\lambda) = 2,\qquad \dim S^+(\lambda) = 1,
\end{equation*}
in $\Real \lambda > 0$. 
Likewise, there exist two unstable eigenvalues $\mu_1^-(\lambda), \mu_2^-(\lambda)$ 
with $\Real \mu < 0$, and one stable eigenvalue $\mu_3^-(\lambda)$ with $\Real \mu > 0$, 
so that the stable (solutions which grow at $-\infty$) and unstable (solutions which decay 
at $-\infty$) manifolds have dimensions
\begin{equation}
	\dim U^-(\lambda) = 1,\qquad \dim S^-(\lambda) = 2.
\end{equation}

\begin{remark}\label{rem:notconsistent}\rm
Notice that, unlike customary situations in the Evans function
literature \cite{AGJ,ZH,GZ,MaZ1,MaZ3,San}, here the dimensions of
the stable (resp.\ unstable) manifolds $S^+$ and $S^-$ (resp. $U^+$
and $U^-$) \emph{do not agree}. 
Under these considerations, we look at the dispersion relation
\begin{equation*}
	\pi_\pm(i\xi) = -i\xi^3 - a_\pm^{-1}(\lambda + Lb_\pm)\xi^2 
	-i\xi-a_\pm^{-1}\,\lambda = 0.
\end{equation*}
For each $\xi \in \R$, the $\lambda$-roots of last equation define
algebraic curves
\begin{equation*}
	\lambda_\pm(\xi) = - i\,a_{\pm}\,\xi
	-\frac{L\,b_\pm\,\xi^2}{1+\xi^2}, \quad \xi \in \R,
\end{equation*}
touching the origin at $\xi =0$. 
Denote $\Lambda$ as the open connected subset of $\C$ bounded 
on the left by the two curves $\lambda_\pm(\xi)$, $\xi \in \R$.
Since $L\,b_\pm>0$ by assumption (A4), the set 
$\{ \Real \lambda \geq 0, \lambda \neq 0\}$ is properly contained in 
$\Lambda$. 
By connectedness the dimensions of $U^\pm(\lambda)$ and
$S^\pm(\lambda)$ do not change in $\lambda \in \Lambda$. 
We define $\Lambda$ as the set of \emph{(not so) consistent splitting}
\cite{AGJ}, in which the matrices $\A_\pm(\lambda)$ remain hyperbolic, 
with not necessarily agreeing dimensions of stable (resp. unstable) manifolds.
\end{remark}

In the low frequency regime $\lambda \sim 0$, we notice, by taking
$\lambda = 0$, that the eigenvalues behave like those of $\A_\pm(0)$. 
If we define
\begin{equation*}
	\begin{aligned}
		\theta_1^+ &:= \half\big( - a_+^{-1}Lb_+ + \sqrt{a_+^{-2}L^2b_+^2 + 4}\,\big),\\
		\theta_1^- &:= \half\big( a_-^{-1}Lb_+ + \sqrt{a_-^{-2}L^2b_-^2 + 4}\,\big), \\
		\theta_3^+ &:= \half\big( a_+^{-1}Lb_+ + \sqrt{a_+^{-2}L^2b_+^2 + 4}\,\big),\\
		\theta_3^- &:= \half\big( - a_-^{-1}Lb_+ + \sqrt{a_-^{-2}L^2b_-^2 + 4}\,\big),
	\end{aligned}
\end{equation*}
as the decay/growth rates for the fast modes (notice that $\theta_j^\pm > 0$, $j=1,3$), 
then the latter are given by
\begin{align*}
	& \mu_2^\pm(0) = 0,  \\
	& \mu_1^-(0) = -\theta^-_1 < 0 < \theta^+_1 = \mu_1^+(0), \\ 
	& \mu_3^+(0) = -\theta^+_3 < 0 < \theta^-_3 = \mu_3^-(0)\ .
\end{align*}
The associated eigenvectors are given by
\begin{equation*}
	V_j^\pm = \begin{pmatrix} b_\pm^{-1}(1-\mu_j^\pm(0)^2) \\ - \mu_j^\pm(0) \\ 1 \end{pmatrix}.
\end{equation*}
Since the highest order coefficient of $\pi_\pm$ as a polynomial in
$\mu$ is different from zero, then $\lambda = 0$ is a regular point
and whence, by standard algebraic curves theory, there exist
convergent series in powers of $\lambda$ for the eigenvalues. 
For low frequency the eigenvalues of $\A_\pm(\lambda)$ have analytic
expansions of the form
\begin{equation}\label{e-values}
	\begin{aligned}
		\mu_2^\pm(\lambda) &= - \frac{\lambda}{a_\pm} + \cO(|\lambda|^2),\\
		\mu_1^\pm(\lambda) &= \pm \theta_1^\pm + \cO(|\lambda|),\\
		\mu_3^\pm(\lambda) &= \mp \theta_3^\pm + \cO(|\lambda|),
	\end{aligned}
\end{equation}
corresponding to a slow varying mode and two fast modes,
respectively, for low frequencies. 
By inspection, the associated eigenvectors can be chosen as
\begin{equation}\label{evectors} 
	V_j^\pm = \begin{pmatrix} b_\pm^{-1}(1-\mu_j^\pm(\lambda)^2) \\ 
						- \mu_j^\pm(\lambda) \\ 1 \end{pmatrix}.
\end{equation}
Notice, in particular, that for this choice of bases, there hold, for  $\lambda \sim 0$,
\begin{equation*}
	V_2^\pm(\lambda) = \begin{pmatrix} 
		\cO(1) \\ \cO(\lambda) \\ \cO(1) \end{pmatrix}, 
	\qquad	
	\qquad V_j^\pm(\lambda) = \cO(1), \;\;\; j=1,3.
\end{equation*}

\begin{lemma}\label{lem:asymmodes} 
Under the same assumptions as in Theorem \ref{theo-main}, for each 
$\lambda \in \Lambda$, the spectral system \eqref{asympsyst} associated to 
the limiting, constant coefficients asymptotic behavior of \eqref{specsyst}, has 
a basis of solutions
\begin{equation}\label{solns-contsys}
	e^{\mu^\pm_j(\lambda) x} V_j^\pm(\lambda), 
	\quad x \gtrless 0, \: j=1,2,3.
\end{equation}
Moreover, for $|\lambda| \sim 0$, we can find analytic representations for 
$\mu_j^\pm$ and $V_j^\pm$, which consist of two slow modes
\begin{equation*}
	\mu_2^\pm(\lambda) = -a_\pm^{-1}\lambda + \cO(\lambda^2),
\end{equation*}
and four fast modes,
\begin{equation*}
	\mu^\pm_1 (\lambda) = \pm \theta^\pm_1 + \cO(\lambda), 
	\qquad
	\mu^\pm_3 (\lambda) = \mp \theta^\pm_3 + \cO(\lambda),
\end{equation*}
with associated eigenvectors \eqref{evectors}.
\end{lemma}
\medskip

\begin{proof} 
The proof is immediate, by directly plugging \eqref{solns-contsys} into 
\eqref{specsyst} and using the previous computations \eqref{e-values}, \eqref{evectors}.
\end{proof}
\medskip

In view of the structure of the asymptotic systems, we are able to
conclude that for each initial condition $x_0> 0$, the solutions to \eqref{specsyst} 
in $x \geq x_0$ are spanned by two growing modes $\{\psi_1^+(x,\lambda), \psi^+_2(x,\lambda)\}$,
and one decaying mode $\{\phi_3^+(x,\lambda)\}$, as $x \to +\infty$,
whereas for each initial condition $x_0 < 0$, the solutions to
\eqref{specsyst} are spanned in $x < x_0$ by two growing modes 
$\{\psi_1^-(x,\lambda), \psi^-_2(x,\lambda)\}$ and one decaying mode
$\{\phi_3^-(x,\lambda) \}$ as $x \to -\infty$.

We rely on the conjugation lemma of \cite{MeZ1} to link such modes
to those of the limiting constant coefficient system \eqref{asympsyst}.
\medskip

\begin{lemma}\label{lem-estmodes} 
Under the same assumptions as in Theorem \ref{theo-main}, for $|\lambda|$ 
sufficiently small, there exist growing $\psi^\pm_j(x,\lambda)$, $j=1,2$, and 
decaying solutions $\phi_3^\pm(x,\lambda)$, in $x \gtrless 0$, of class $C^1$ 
in $x$ and analytic in $\lambda$, satisfying
\begin{equation*}
	\begin{aligned}
		\psi^\pm_j(x,\lambda) &= e^{\mu_j^\pm(\lambda)} V_j^\pm(\lambda)
		(I+ \cO(e^{-\eta|x|})), \quad j=1,2,\\
		\phi^\pm_3(x,\lambda) &= e^{\mu_3^\pm(\lambda)} V_3^\pm(\lambda) 
		(I+ \cO(e^{-\eta|x|})),
	\end{aligned}
\end{equation*}
where $\eta>0$ is the decay rate of the traveling wave, and $\mu_j^\pm$ and 
$V_j^\pm$ are as in Lemma \ref{lem:asymmodes} above.
\end{lemma}
\medskip

\begin{proof}
This a direct application of the conjugation lemma of \cite{MeZ1}
(see also the related gap lemma in \cite{GZ,ZH,MaZ1,MaZ3}).
\end{proof}
\medskip

As a corollary, and in order to sum up the observations of this
section, we make note that for $\lambda \sim 0$, the
solutions to \eqref{specsyst} in $x \geq x_0 > 0$ are spanned by
\begin{align}
\psi_1^+(x,\lambda) &= e^{(\theta^+_1 + \cO(|\lambda|))x}
\, V_1^+(\lambda)(I + \cO(e^{-\eta|x|})), \quad \;\;\;\;\text{(fast growing)}, \label{gfast+}\\
\psi_2^+(x,\lambda) &= e^{(-\lambda/a_+ + \cO(|\lambda|^2))x} \,
V_2^+(\lambda)(I + \cO(e^{-\eta|x|})), \;\text{(slowly
growing)},\label{gslow+}\\
\phi_3^+(x,\lambda) &= e^{(-\theta^+_3 + \cO(|\lambda|))x} \,
V_3^+(\lambda)(I + \cO(e^{-\eta|x|})), \quad \;\;\text{(fast
decaying)}. \label{dfast+}
\end{align}
Likewise, all the solutions for $x \leq x_0 < 0$ comprise the modes
\begin{align}
\psi_1^-(x,\lambda) &= e^{(-\theta^-_1 + \cO(|\lambda|))x} \,
V_1^-(\lambda)(I + \cO(e^{-\eta|x|})), \quad \;\;\text{(fast growing)}, \label{gfast-}\\
\psi_2^-(x,\lambda) &= e^{(-\lambda/a_- + \cO(|\lambda|^2))x} \,
V_2^-(\lambda)(I + \cO(e^{-\eta|x|})), \;
\text{(slowly growing)},\label{gslow-}\\
\phi_3^-(x,\lambda) &= e^{(\theta^-_3 + \cO(|\lambda|))x}
V_3^-(\lambda) \,(I + \cO(e^{-\eta|x|})), \quad
\;\;\;\;\,\text{(fast decaying)}. \label{dfast-}
\end{align}
The analytic coefficients $V^\pm_j(\lambda)$ are given by \eqref{evectors}.

\subsection{Solutions near $x \sim 0$}

Our goal now is to analyze system \eqref{specsyst} close to the singularity $x=0$. 
For concreteness, let us restrict the analysis to the case $x>0$. 
We introduce a ``stretched'' variable $\xi$ as follows: fix $\epsilon_0 > 0$ and let
\begin{equation*}
    \xi = \int_{\epsilon_0}^{x}\frac{dz}{a(z)},
\end{equation*}
so that $\xi(\epsilon_0) = 0$, and $\xi \to +\infty$ as $x \to 0^+$.
Under this change of variables we get
\begin{equation*}
    u' = \frac{du}{dx} = \frac{1}{a(x)}\frac{du}{d\xi} =
    \frac{1}{a(x)}\dot{u},
\end{equation*}
after denoting $\;\dot{ }$ $= d/d\xi$.
In the stretched variables, system \eqref{eq:linLaplace2} becomes
\begin{equation*} 
	\dot{W} = \tilde\A(\xi,\lambda) W
	\qquad\textrm{where}\quad
	\tilde\A(\xi,\lambda) := \begin{pmatrix}
	-\omega & 0 & L \\ \tilde a\, \tilde b & 0 & -\tilde a \\ 
	0 & -\tilde a & 0 \end{pmatrix},
\end{equation*}
and functions $\omega, \tilde a, \tilde b$ are defined by
\begin{equation*}
		\omega(\xi):= \lambda + a'(x(\xi)) + L\,b(x(\xi)),\qquad
		\tilde a(\xi):= a(x(\xi)), \qquad \tilde b(\xi):= b(x(\xi)).
\end{equation*}
Note that from \eqref{eq:profdiff}, for small frequencies $\lambda \sim 0$, 
and choosing $ 0 < \epsilon_0 \ll 1$ sufficiently small we have the uniform bound
\begin{equation*}
	\Real \omega(\xi) \sim \Real \omega(0) = 
	\eta: = \Real \lambda + a'(0) + L\,b(0) > 0,
\end{equation*}
for all $\xi \in [0,+\infty)$.
In addition, we have
\begin{equation*}
	\omega_\xi = \tilde a(\xi) (a''(x(\xi)) + Lb'(x(\xi))) = \cO(|\tilde a(\xi)|).
\end{equation*}
Next, we apply the transformation $Z := {\mathbb{L}} W$ where
\begin{equation*}
	{\mathbb{L}} := \begin{pmatrix} 1 & 0 & -L/\omega \\ 0 & 1 & 0 \\ 0 & 0 & 1 \end{pmatrix}
	\qquad\textrm{and}\qquad
	{\mathbb R}:={\mathbb{L}}^{-1}
		= \begin{pmatrix} 1 & 0 & L/\omega \\ 0 & 1 & 0 \\ 0 & 0 & 1 \end{pmatrix}.
\end{equation*}
Since
\begin{equation*}
	|\dot{{\mathbb{L}}}\R| = |{\mathbb{L}}\dot{\R}| = \cO(|\tilde a|),\qquad\textrm{and}\qquad
	\dot{{\mathbb{L}}} =  \begin{pmatrix}  0& 0& L\omega_\xi / \omega^2 \\
		0 & 0& 0\\	0& 0& 0	\end{pmatrix} = \tilde a\, \cO(1),
\end{equation*}
we obtain a block-diagonalized system at leading order of the form
\begin{equation}\label{eq:strechtedsyst} 
	\dot{Z} = \begin{pmatrix} -\omega & 0 \\ 0&
	0 \end{pmatrix}Z + \tilde a\, \Theta(\xi) Z,
\end{equation}
where
\begin{equation*}
	\Theta = \begin{pmatrix} 0 & L/\omega & L(a''+L\,b')/\omega^2\\ 
	\tilde b & 0 & -1 + L\,\tilde b / \omega \\ 0 & -1 & 0
	\end{pmatrix}
\end{equation*}
is uniformly bounded. 
The blocks $-\omega I$ and $0$ are clearly spectrally separated and the 
error is of order $\cO(|\tilde a(\xi)|) \to 0$ as $\xi \to +\infty$. 
System \eqref{eq:strechtedsyst} has the form \eqref{eq:blockdiag} of 
Appendix \ref{firstAppendix} (block-diagonal  at leading order) and satisfies 
the hypotheses of the  pointwise reduction lemma (see Proposition \ref{pwrl} below). 
In our case, there is no dependence on a parameter $\epsilon$, 
$M_2 =-\omega I$, $M_1 \equiv 0$ and the pointwise error is
$\delta(\xi) = \tilde a (\xi)$, with constant spectral gap $\eta$. 

Hence, there exist analytic transformations $\Phi_j(\xi,\lambda)$, $j =1,2$, 
satisfying the  pointwise bound \eqref{ptwise}, for which the graphs
$\{(Z_1,\Phi_2(Z_1))\}$, $\{(\Phi_1(Z_2),Z_2)\}$ are invariant under
the flow of \eqref{eq:strechtedsyst}. 
We now take a closer look at the pointwise error \eqref{ptwise}.
\medskip

\begin{lemma}\label{lem-pttracking}
For the stretched system and for low frequency $\lambda \sim 0$,
there holds 
\begin{equation} \label{eq:errorOa}
	|\Phi_j(\xi,\lambda)|\le C\, \tilde a(\xi), \quad j=1,2,
\end{equation}
provided that
\begin{equation} \label{newcondition} 
	Lb(0) + 2a'(0) > 0.
\end{equation}
\end{lemma}
\medskip

\begin{proof}
From Proposition \ref{pwrl}, there holds the pointwise bound \eqref{ptwise}, 
namely
\begin{equation*}
	|\Phi_j(\xi,\lambda)| \leq C \int_0^\xi e^{-\eta(\xi-y)} \tilde a(y) \, dy,
\end{equation*}
which in terms of the original variables looks like
\begin{equation*}
	|\tilde \Phi_j(x,\lambda)| := |\Phi_j(\xi(x),\lambda)| \leq C \int_{\epsilon_0}^x
	\exp \Big(\eta \int_x^{\tilde x} \frac{dz}{a(z)}\Big) \, d\tilde x.
\end{equation*}
Since for $z$ small, $a(z) \sim a'(0)z$, we get
\begin{align*}
	|\tilde \Phi_j(x,\lambda)| &\lesssim 
	\int_{\epsilon_0}^x \exp \Big( \eta \int_x^{\tilde x} \frac{dz}{a'(0)z} \Big) \, d \tilde x
	= \frac{C\, a'(0)}{\eta + a'(0)} \big( x - \epsilon_0 (x/\epsilon_0)^{\eta/|a'(0)|}\big) \\
	&\leq \frac{C\, a'(0) x}{\eta + a'(0)} \sim \frac{C\,a(x)}{\eta + a'(0)},
\end{align*}
in view of $0 < x < \epsilon_0$, and as long as $a'(0) + \eta > 0$.
Since $\eta = \Real \lambda + a'(0) + Lb(0)$, 
condition \eqref{newcondition} implies \eqref{eq:errorOa} for small $\lambda$.
\end{proof}
\medskip

\begin{remark}\rm
Notice that \eqref{newcondition} is a stronger condition than \eqref{eq:profdiff}, 
which is inherited by the existence result of \cite{LMS1} or Theorem \ref{existencethm}. 
Notably, this new condition \eqref{newcondition} holds if we assume (A$5_2$).
\end{remark}

In view of the pointwise error bound \eqref{eq:errorOa} of order $\cO(a)$ and by the 
pointwise reduction lemma (see Proposition \ref{pwrl} and Remark \ref{rem:reduced} below), 
we can separate the flow into slow and fast coordinates. 
Indeed, after proper transformations we separate the flows on the reduced manifolds of form
\begin{equation}\label{reduced-eq}
	\begin{aligned}
		\dot{Z_1} &= - \omega\,Z_1 + \cO(\tilde a) Z_1,\\
		\dot{Z_2} &= \cO(\tilde a) Z_2.
	\end{aligned}
\end{equation}
Observe that the $Z_1$ modes decay to zero as $\xi \to +\infty$, in view of
\begin{equation*}
	e^{-\int_0^\xi \omega(z) \, dz} \lesssim e^{-(\Real \lambda + \half \eta) \xi} \to 0,
\end{equation*} 
as $\xi \to +\infty$. 
These fast decaying modes correspond to fast decaying to zero solutions when $x
\to 0^+$ in the original $u$-variable. 
The $Z_2$ modes comprise slow dynamics of the flow as $x \to 0^+$.
\medskip

\begin{proposition} \label{prop:smallep} 
Under assumptions \eqref{A0} - \eqref{A4}, and $\mathrm{(A}5_2\mathrm{)}$, 
there exists $0 < \epsilon_0 \ll 1$ sufficiently small, such that, in the small frequency 
regime $\lambda \sim 0$, the solutions to the spectral system \eqref{specsyst} in 
$(-\epsilon_0,0) \cup (0,\epsilon_0)$ are spanned by fast modes
\begin{equation*}\label{w2} 
	w_2^\pm(x,\lambda) 
	= \begin{pmatrix} u_2^\pm\\  q_2^\pm \\  p_2^\pm \end{pmatrix} 
	=\begin{pmatrix} Z_1(x)\\ 0 \\ 0 \end{pmatrix} (1+\cO(a(x))), 
	\qquad \pm\epsilon_0 \gtrless x \gtrless 0,
\end{equation*} 
where $Z_1$ is the mode of \eqref{reduced-eq}, decaying to zero as $x \to 0^\pm$, 
and slowly varying modes
\begin{equation*}
	z_j^\pm(x,\lambda) 
	= \begin{pmatrix} u_j^\pm \\ q_j^\pm \\  p_j^\pm \end{pmatrix},  
	\qquad \pm\epsilon_0 \gtrless x \gtrless 0, \quad j=1,3,\label{z13}
\end{equation*}
with bounded limits as $x \to 0^\pm$.
Moreover, the fast modes defined above decay as
\begin{equation}\label{decayu2} 
	u_2^\pm \sim |x|^{\nu} \to 0, \qquad
	\begin{pmatrix} q_2^\pm \\ p_2^\pm \end{pmatrix} 
	\sim \cO(|x|^\nu a(x)) \to 0,
\end{equation}
as $x \to 0^\pm$, where $\nu: =(\Real\lambda + a'(0)+Lb(0))/|a'(0)|$.
\end{proposition}
\medskip

\begin{proof}
This is a direct consequence of applying our pointwise tracking lemma (Lemma
\ref{lem-pttracking}) to the reduced system \eqref{reduced-eq}. 
The claimed estimate \eqref{decayu2} for $u$ follows in the same way as
done in Lemma \ref{lem:tech}.
\end{proof}

\subsection{Decaying modes}

We next derive explicit representation formulae for the resolvent
kernel $\cG_\lambda(x,y)$ using the classical construction in terms
of decaying solutions of the homogeneous spectral problem, matched
across the singularity by appropriate jump conditions at $x=y$. 
The novelty of our approach circumvents the inconsistency between the
number of decaying modes at $\pm \infty$. 
In this section we describe how to construct all decaying solutions at each 
side of the singularity with matching dimensions.

Choose $\epsilon_0>0$ small enough so that the representations of
the solutions of Proposition \ref{prop:smallep} hold. We are going
to construct two decaying modes $W^+_j$, $j =1,2$ at $+\infty$, and
one decaying mode $W^-_3$ at $-\infty$. 
For that purpose, we choose the decaying mode at $-\infty$ as
\begin{equation}\label{eq:defW3} 
	W^-_3(x,\lambda) := \begin{cases} 
	\phi^-_3(x,\lambda), & x < -\epsilon_0, \\ 
	(\gamma_1 z_1^- + \gamma_3 z_3^- + \gamma_2 w_2^-) (x,\lambda), 
	& -\epsilon_0 <x<0.
	\end{cases}
\end{equation}
where the coefficients $\gamma_j = \gamma_j(\lambda)$ are analytic
in $\lambda$ and such that $W^-_3$ is of class $C^1$ in all $x < 0$.

To select the decaying modes at $+\infty$, consider
\begin{equation}\label{eq:defW2+} 
	W_2^+(x,\lambda) := \begin{cases} 0, & x > 0, \\
	w_2^-(x,\lambda), &-\epsilon_0<x < 0,\\
	(\kappa_1 \psi_1^- + \kappa_2 \psi_2^- + \kappa_3
	\phi_3^-) (x,\lambda), & x< -\epsilon_0 .
	\end{cases}
\end{equation}
where $w_2^-$ is the vanishing at $x=0$ solution in \eqref{w2} 
(the solution is, thus, continuous at $x=0$), and the coefficients
$\kappa_j = \kappa_j(\lambda)$ are analytic in $\lambda$,
and such that the matching is of class $C^1$ a.e. in $x$.

Finally, we define
\begin{equation}\label{eq:defW1+} 
	W_1^+(x,\lambda) := \begin{cases}
	\phi^+_3(x,\lambda), & x > \epsilon_0, \\ (\alpha_1 z_1^+ + \alpha_3
	z_3^+ + \alpha_2 w_2^+) (x,\lambda), & 0 < x < \epsilon_0,
	\\(\beta_1 z_1^- + \beta_3 z_3^- + \beta_2 w_2^-) (x,\lambda), &
	-\epsilon_0<x<0 \\
	(\delta_1 \psi_1^- + \delta_2 \psi_2^- + \delta_3 \phi_3^-)
	(x,\lambda), & x< -\epsilon_0 .
	\end{cases}
\end{equation}
as the other decaying mode at $+\infty$, with analytic coefficients
$\alpha_j$, $\beta_j$, $\delta_j$ in $\lambda$, and $W^+_1$ is of
class $C^1$ a.e. in $x$.

\begin{remark}\rm
A similar definition of two decaying modes $W^-_3, W^-_2$ at
$-\infty$ and one decaying mode $W^+_1$ at $+\infty$, on the
positive side of the singularity, is clearly available. See Figure \ref{twoevansfigure}.
\end{remark}

\begin{figure}[h]
\includegraphics[width=6cm]{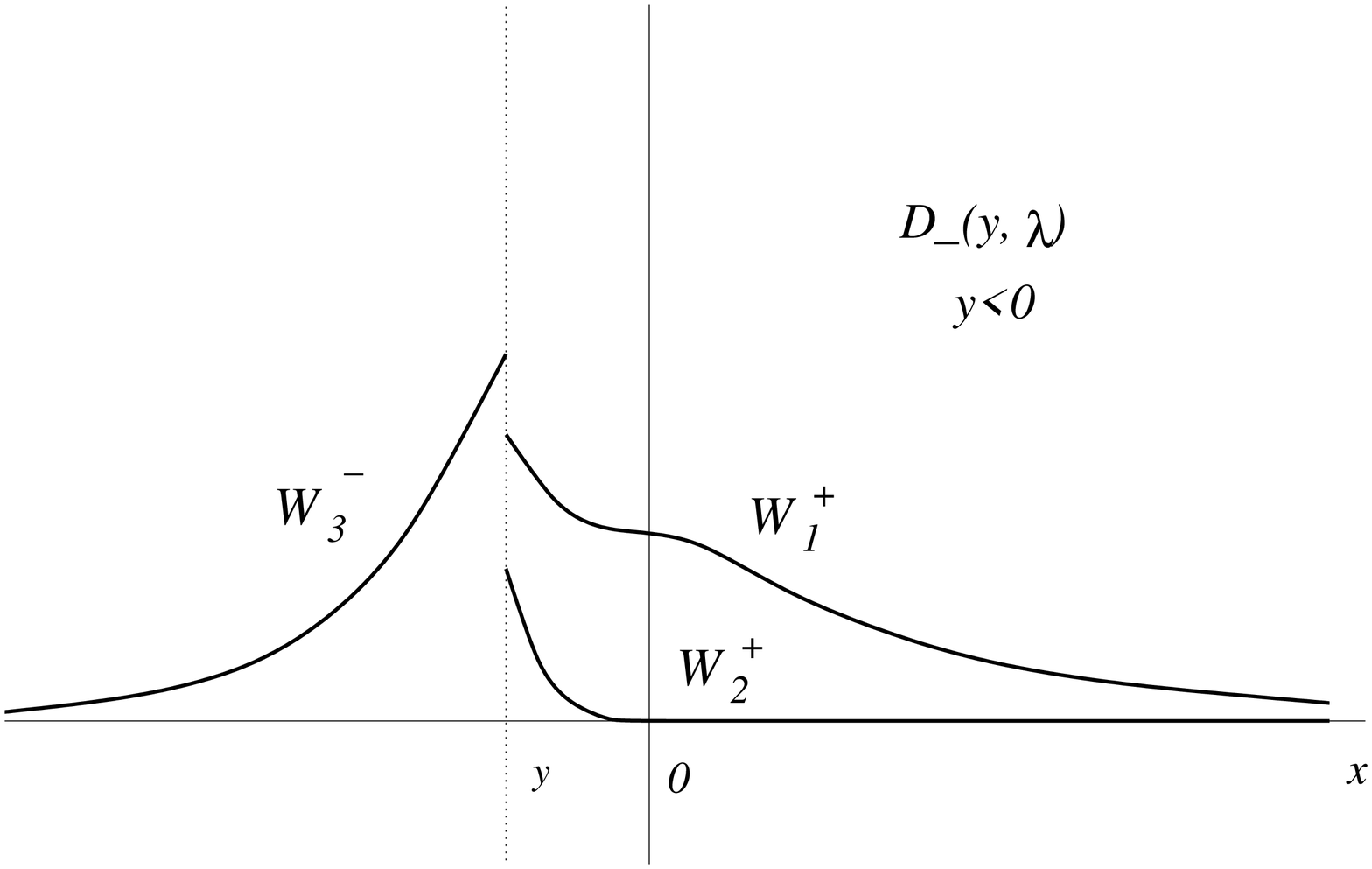}\quad
\includegraphics[width=6cm]{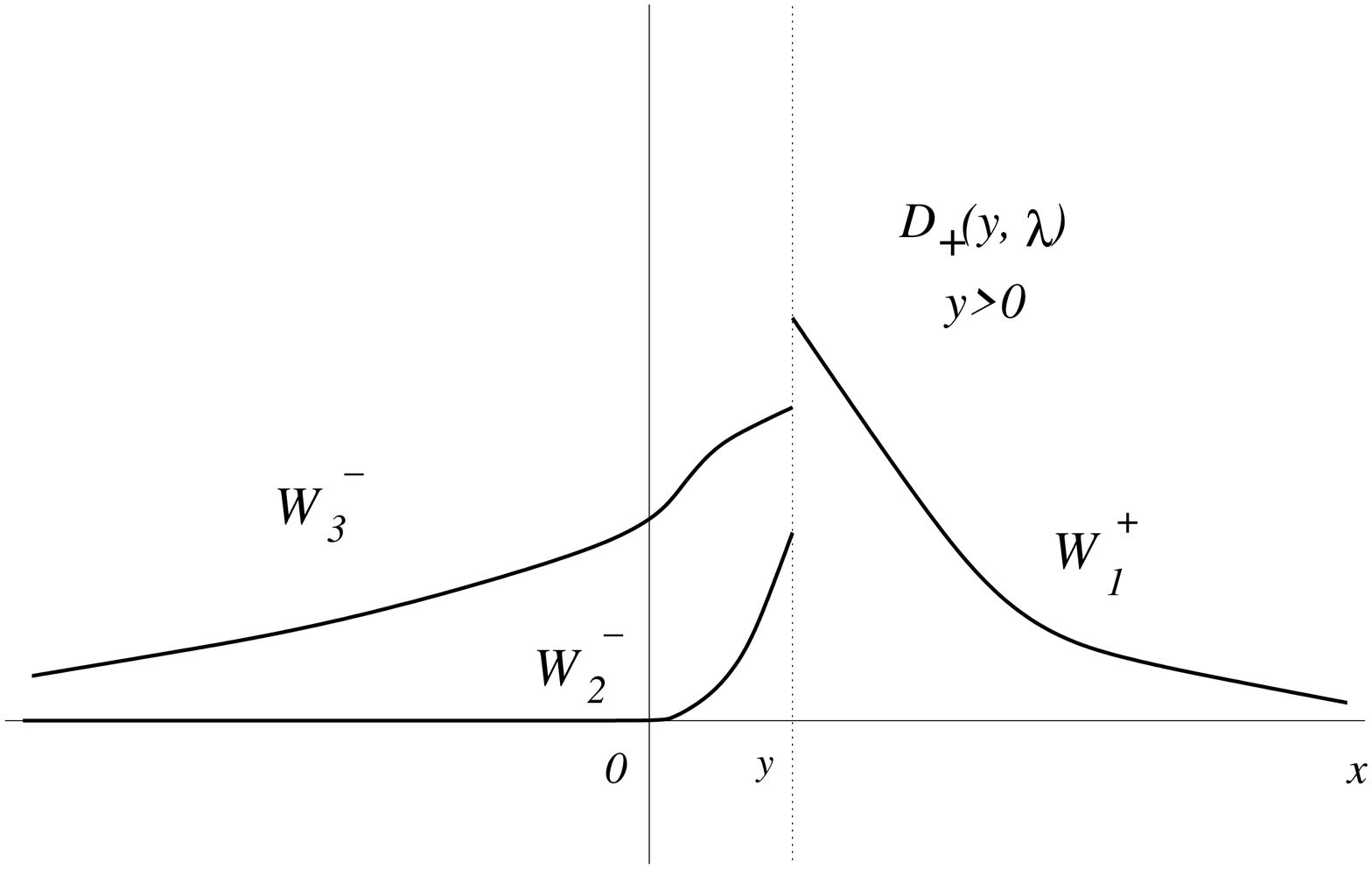}
\caption{\footnotesize \textit{Two Evans functions:} 
Representation of the decaying modes at $\pm\infty$ for $y \gtrless 0$. 
The first picture (left) considers the case $y < 0$. 
The sole decaying mode $W_3^-$ at $- \infty$ is represented as the fast decaying
solution on the left; the decaying modes at $+\infty$ are two: the exponentially fast 
decaying solution $W_1^+$, and the constructed mode $W_2^+$, which is identically 
zero for $x > 0$ and matched across the singularity to the solution which decays to zero 
as $x \to 0^-$ in the region $(y,0)$. 
This provides a full set of decaying modes for $y < 0$. 
A symmetric construction for the $y > 0$ case is depicted in the second picture (right).}
\label{twoevansfigure} 
\end{figure}

\subsection{Two Evans functions}

We first define two related Evans functions
\begin{equation*}
	D_\pm(y,\lambda) := \det (W_1^+ \; W_2^{\mp} \;
	W^-_3)(y,\lambda), \qquad \mbox{for  } y \gtrless 0,
\end{equation*} 
where $W_j^\pm = (u_j^\pm,q_j^\pm,p_j^\pm)^\top$ is defined as 
above (see \eqref{eq:defW1+},\eqref{eq:defW2+}, and \eqref{eq:defW3}).

We first observe the following simple properties of $D_\pm$.
\medskip

\begin{lemma}\label{lem-Evansfns}
For $\lambda$ sufficiently small, we have
\begin{equation}\label{Evansfns2}
	\begin{aligned}
		D_\pm(y,\lambda)&=-a(y)^{-1}\lambda[u]
		\det\begin{pmatrix}q_1^+&q_2^\mp\\
		p_1^+&p_2^\mp\end{pmatrix}_{|_{\lambda=0}}
		+ \cO(|\lambda|^2),
	\end{aligned}
\end{equation} 
where $[u] = u_+ - u_-$.
\end{lemma}
\medskip

\begin{proof} Let us consider \eqref{Evansfns2} for $D_-$.
By our choice, at $\lambda=0$, we can take
\begin{equation}\label{gchoice}
	W_1^+ (x,0)= W_3^- (x,0) = \bar W'(x)
\end{equation} 
where $\bar W$ is the shock profile. 
By Leibnitz' rule, we first compute
\begin{equation*}
	\begin{aligned} 
		\partial_\lambda D_-(y,0)&= \det\Big(\partial_\lambda
		W_1^+,W_2^+,W_3^-\Big)_{|_{\lambda=0}}\\&\qquad +\det\Big(
		W_1^+,\partial_\lambda W_2^+,W_3^-\Big)_{|_{\lambda=0}}
		+\det\Big(W_1^+,W_2^+,\partial_\lambda
		W_3^-\Big)_{|_{\lambda=0}}
	\end{aligned}
\end{equation*}
where, by using \eqref{gchoice}, the second term on the right hand side vanishes 
and the first and third terms can be grouped together, yielding
\begin{equation}\label{der-D}
		\partial_\lambda D_-(y,0)=\det\Big(W_1^+,W_2^+,\partial_\lambda
		W_3^--\partial_\lambda W_1^+\Big)_{|_{\lambda=0}}.
\end{equation}
Since $W_j^\pm(\cdot,\lambda)$ satisfies \eqref{eqW}, 
$\partial_\lambda W_1^+(x,0)= (\partial_\lambda
u_1^+,\partial_\lambda q_1^+,\partial_\lambda p_1^+)$
satisfies
\begin{equation*}
 	\Theta (\partial_\lambda W_1^+)' 
	= \A(x,0)\partial_\lambda W_1^+(x,0) 
 	+ \partial_\lambda\A(x,0)W_1^+(x,0), 
\end{equation*}
which directly gives
\begin{equation}\label{eq-W1}
 	(a\,\partial_\lambda u_1^+)' = - L(\partial_\lambda q_1^+)' - \bar u'.
\end{equation}
Likewise, $\partial_\lambda W_3^-(x,0) = (\partial_\lambda
u_3^-,\partial_\lambda q_3^-,\partial_\lambda p_3^-)$ satisfies
\begin{equation}\label{eq-W3}
	(a\,\partial_\lambda u_3^-)' = - L(\partial_\lambda q_3^-)' - \bar u'.
\end{equation}
Integrating equations \eqref{eq-W1} and \eqref{eq-W3} from $+\infty$
and $-\infty$, respectively, with use of boundary conditions
$\partial_\lambda W_1^+(+\infty,0) = \partial_\lambda W_3^-(-\infty,0) =0$,  
we obtain
\begin{equation*}
 		a\,\partial_\lambda u_1^+= - L\partial_\lambda q_1^+ - \bar u + u_+,
	\qquad\textrm{and}\qquad
		a\,\partial_\lambda u_3^-= - L\partial_\lambda q_3^- - \bar u + u_-.
\end{equation*} 
Thus
\begin{equation}\label{col3}
 	a(\partial_\lambda u_3^--\partial_\lambda u_1^+)
	= - L(\partial_\lambda q_3^--\partial_\lambda q_1^+) -[u].
\end{equation}
Meanwhile, since $W_j^+$, $j=1,2$ satisfy the equation \eqref{eqW}
and thus $(au)' = -Lq'$ with $W_j^+(+\infty,\lambda)=0$, we integrate the
latter equation, yielding
\begin{equation}\label{col1-2} 
	a u_j^+ = -L q_j^+, \qquad \mbox{for  }j=1,2.
\end{equation}
Using estimates \eqref{col1-2} and \eqref{col3}, we can now compute
the $\lambda$-derivative \eqref{der-D} of $D_\pm$ at $\lambda=0$ as
\begin{equation}\label{derD-}
	\begin{aligned}
\partial_\lambda D_-(y,0)&=\det\begin{pmatrix}u_1^+&u_2^+&\partial_\lambda
u_3^--\partial_\lambda u_1^+\\q_1^+&q_2^+&\partial_\lambda
q_3^--\partial_\lambda q_1^+\\p_1^+&p_2^+&\partial_\lambda
p_3^--\partial_\lambda p_1^+\end{pmatrix}\\
&=\det\begin{pmatrix}u_1^+&u_2^+&\partial_\lambda
u_3^--\partial_\lambda
u_1^+\\0&0&-[u]/L\\p_1^+&p_2^+&\partial_\lambda
p_3^--\partial_\lambda p_1^+\end{pmatrix}\\
&=L^{-1}[u]
\det\begin{pmatrix}u_1^+&u_2^+\\p_1^+&p_2^+\end{pmatrix}.
	\end{aligned}
\end{equation}
Applying again relation \eqref{col3}, we obtain \eqref{Evansfns2}. 

Similarly, for $D_+$ we obtain
\begin{equation}\label{derD+}
	\begin{aligned}
		\partial_\lambda D_+(y,0) &=-L^{-1}[u]
		\det\begin{pmatrix}u_1^+&u_2^-\\p_1^+&p_2^-
		\end{pmatrix}
	\end{aligned}
\end{equation}
from which the conclusion follows.
\end{proof}
\medskip

Since there are two different Evans functions for $y \gtrless 0$, we need to be
sure if one vanishes to order one (part of  the stability criterion), then the other 
does too.
Such property, content of the following Lemma, guarantees  that pole terms are 
the same on $y<0$ and $y>0$. 
\medskip

\begin{lemma} 
Defining the Evans functions 
\begin{equation}\label{Evansfns-def}
	D_\pm(\lambda): = D_\pm(\pm 1, \lambda),
\end{equation} 
we then have $D_+(\lambda) = m D_-(\lambda) +\cO(|\lambda|^2)$
where $m$ is some nonzero factor.
\end{lemma}
\medskip

\begin{proof} 
Since $W_1^+(x)=\bar W'$ is a nonvanishing, bounded solution of the ODE \eqref{eqW}, 
we must have $W_1^+(1) = m_1 W_1^+(-1)$ for some $m_1$ nonzero. 
Meanwhile, Proposition \ref{prop:smallep} gives
\begin{equation*}
	\begin{pmatrix}u_2^\pm\\p_2^\pm\end{pmatrix} =
	\begin{pmatrix}|x|^{\nu}\\0\end{pmatrix} + \cO(|x|^\nu a(x)),
\end{equation*} 
as $x\to 0$, where $\nu = (a'(0)+Lb(0))/|a'(0)|$.
Thus, smoothness of $a$ near zero guarantees an existence of $\epsilon_1,\epsilon_2$ 
near zero such that
\begin{equation*}
	\begin{pmatrix}u_2^+\\p_2^+\end{pmatrix}_{x=-\epsilon_1} =
	\begin{pmatrix}u_2^-\\p_2^-\end{pmatrix}_{x=\epsilon_2}.
\end{equation*} 
This together with the fact that $W_2^\pm$ are solutions of the ODE \eqref{eqW} yields
\begin{equation*}
	\begin{pmatrix}	u_2^+	\\	p_2^+	\end{pmatrix}_{x=-1}	=m_2
	\begin{pmatrix}	u_2^-	\\	p_2^-	\end{pmatrix}_{x=1}
\end{equation*} 
for some $m_2$ nonzero. 
Putting these estimates into \eqref{derD-} and \eqref{derD+} and using analyticity 
of $D_\pm$ in $\lambda$ near zero, we easily obtain the conclusion.
\end{proof}

\section{Resolvent kernel bounds in low--frequency regions} 
\label{sec:reskbounds}

In this section, we shall derive pointwise bounds on the resolvent kernel 
$G_{\lambda}(x,y)$ in low-frequency regimes, that is, $|\lambda| \to 0$. 
For definiteness, throughout this section, we consider only the case $y<0$. 
The case $y>0$ is completely analogous by symmetry.

We solve \eqref{eq:resolker} with the jump conditions at $x=y$:
\begin{equation}\label{eq:conduy} 
	[\cG_\lambda(.,y)] 
	= \begin{pmatrix}a(y)^{-1} &0&0\\0& 1&0\\0&0&1\end{pmatrix}
\end{equation}
Meanwhile, we can write $\cG_\lambda(x,y)$ in  terms of decaying
solutions at $\pm \infty$ as follows
\begin{equation}\label{eq:formcolu} 
	\cG_\lambda(x,y) = \begin{cases}
	W_1^+(x,\lambda) C_1^+(y,\lambda) + W_2^+(x,\lambda) C_2^+(y,\lambda), & x>y,\\
	- W_3^-(x,\lambda) C_3^-(y,\lambda), & x<y\end{cases}
\end{equation}
where $C_j^\pm = (C_{jk}^\pm)_{k=1,2,3}$ are row vectors. 
We compute the coefficients $C^\pm_j$ by means of the transmission conditions
\eqref{eq:conduy} at $y$. 
Therefore, solving by Cramer's rule the system
\begin{equation*}
	\begin{pmatrix} W_1^+ & W_2^+ & W_3^- \end{pmatrix}
	\begin{pmatrix}C_1^+ \\ C_2^+ \\ C_3^- \end{pmatrix}\Bigr|_{(y,\lambda)} =
	\begin{pmatrix}a(y)^{-1}
	&0&0\\0& 1&0\\0&0&1\end{pmatrix},
\end{equation*}
we readily obtain,
\begin{equation*}
	\begin{pmatrix}C_1^+ \\ C_2^+ \\ C_3^- \end{pmatrix}
	= D_-(y,\lambda)^{-1}
	\begin{pmatrix} W_1^+ & W_2^+ & W_3^-\end{pmatrix}^{adj}\bigr|_{(y,\lambda)}
	\begin{pmatrix} a(y)^{-1} &0&0\\0& 1&0\\0&0&1\end{pmatrix}
\end{equation*} 
where $M^{adj}$ denotes the adjugate matrix of a matrix $M$. 
For example, 
\begin{align}
	C_{11}^+(y,\lambda) &= a(y)^{-1}D_-(y,\lambda)^{-1} \left| \begin{matrix} q_2^+ & q_3^- \\
	p_2^+ & p_3^- \end{matrix}\right|(y,\lambda), \label{C1+}\\
	C_{21}^+(y,\lambda) &= a(y)^{-1}  D_-(y,\lambda)^{-1}\left| \begin{matrix} q_3^- & q_1^+  \\
	p_3^- & p_1^+ \end{matrix}\right|(y,\lambda), \label{C2+}\\
	C_{31}^-(y,\lambda) &= a(y)^{-1}  D_-(y,\lambda)^{-1}\left| \begin{matrix} q_1^+ & q_2^+  \\
	p_1^+ & p_2^+ \end{matrix}\right|(y,\lambda).\label{C3-}
\end{align} 
Here, note that these are only coefficients that are
possibly singular as $y$ near zero because of singularity in the
first column of the jump-condition matrix \eqref{eq:conduy}.

We then easily obtain the following.
\medskip

\begin{lemma}\label{lem-estCnear0} 
For $y$ near zero, we have
\begin{equation}\label{est-C13}
	\begin{aligned}
		C_1^+(y,\lambda) &=~~~\frac {1}{\lambda}[u]^{-1}(1,\;-L,\;0)+ \cO(1),\\
		C_3^-(y,\lambda) &= -\frac {1}{\lambda}[u]^{-1} (1,\;-L,\;0)+\cO(1),
	\end{aligned}
\end{equation}
and 
\begin{equation}\label{est-C2}
	\begin{aligned} 
		C_2^+(y,\lambda)&=a(y)^{-1}|y|^{-\nu}\cO(1),
	\end{aligned}
\end{equation}
where $\nu$ is defined as in Proposition \ref{prop:smallep} and $\cO(1)$ is 
a uniformly bounded function, possibly depending on $y$ and $\lambda$.
\end{lemma}
\medskip

\begin{proof} 
It suffices to estimate $C_{j1}^\pm$ when the singularity plays a role.
Recalling \eqref{Evansfns2} and \eqref{C3-}, we can estimate $C_{31}^-(y,\lambda)$ as
\begin{equation*}
	C_{31}^-(y,\lambda) = - \frac1{\lambda[u]}
	\det\begin{pmatrix}q_1^+&q_2^+\\p_1^+&p_2^+\end{pmatrix}^{-1}_{|_{\lambda=0}}
	\Big[\left| \begin{matrix} q_1^+ & q_2^+  \\
	p_1^+ & p_2^+ \end{matrix}\right|(y,0) + \cO(\lambda)\Big]
	= - \frac1{\lambda[u]} + \cO(1),
\end{equation*} 
where $\cO(1)$ is uniformly bounded since $a(y)D_-(y,\lambda)$ and normal modes 
$W^\pm_j$ are all bounded uniformly in $y$ near zero. 
This yields the bound for $C_{31}^-$ as claimed. 
The bound for $C_{11}^+$ follows similarly, noting that $W_3^- \equiv W_1^+$ at $\lambda=0$.

For the estimate on $C_2^+$, we first observe that by view of
\eqref{Evansfns2} and the estimate \eqref{decayu2} on $u_2^+$,
\begin{equation*} 
	D_-(y,\lambda) \ge c\,\lambda\, |y|^{\nu},
\end{equation*} 
for some $c>0$. 
This together with the fact that $W_3^- \equiv W_1^+$ at $\lambda=0$ yields 
the estimate for $C_2^+$ as claimed.
\end{proof}
\medskip

\begin{proposition}[Resolvent kernel bounds as $|y|\to 0$]\label{prop-nearzero} 
Assume \eqref{A0} - \eqref{A5k}. For $y$ near zero, there hold
\begin{equation}\label{G0-est1} 
	\cG_\lambda(x,y) ={\lambda^{-1}}[u]^{-1}\bar W'(1,\;-L,\;0) + \cO(e^{-\eta |x|})
\end{equation} 
for $y<0<x$, and
\begin{equation}\label{G0-est2}
	\cG_\lambda(x,y) = {\lambda^{-1}}[u]^{-1}\bar W' (1,\;-L,\;0)+
	\cO(1) \Big(1+\frac{|x|^{\nu}}{a(y)|y|^\nu}\Big)
\end{equation} 
for $y<x<0$, and
\begin{equation*}
	\cG_\lambda(x,y) ={\lambda^{-1}}[u]^{-1}\bar W' (1,\;-L,\;0)+ \cO(e^{-\eta |x|})
\end{equation*} 
for $x<y<0$, for some $\eta>0$.
Similar bounds can be obtained for the case $y>0$.
\end{proposition}
\medskip

\begin{proof} 
For the case $y<0<x$, using \eqref{est-C13} and recalling that 
$W_1^+(x) = \bar W'  + \cO(\lambda)e^{-\eta |x|}$ and 
$W_2^+(x) \equiv 0$, we have
\begin{equation*}
	\cG_\lambda(x,y) = W_1^+(x)\, C_1^+(y) 
		= \Big(\bar W' +\cO(\lambda)e^{-\eta |x|}\Big)
		\Big(\frac{1}{\lambda[u]}(1,\;-L,\;0) + \cO(1)\Big),
\end{equation*}
yielding \eqref{G0-est1}. In the second case $y<x<0$, from the formula
\eqref{eq:formcolu} projected on the first component, we have
\begin{equation*}
	C_1^+(y,\lambda)u_1^+(x,\lambda) + C_2^+(y,\lambda)u_2^+(x,\lambda)
\end{equation*}	
where the first term contributes ${\lambda^{-1}}[u]^{-1}\bar W' + \cO(1)$ as in the 
first case, and the second term is estimated by \eqref{est-C2} and \eqref{decayu2}.

Finally, we estimate the last case $x<y<0$ in a same way as done in the first case, 
noting that $y$ is still near zero and $W_3^-(x) = \bar W' + \cO(\lambda)e^{-\eta |x|}$. 
\end{proof}

Next, we derive pointwise bounds of $G_\lambda(x,y)$ in regions $|y|\to +\infty$. 
Note however that the representations \eqref{eq:formcolu} and above estimates fail 
to be useful in the $y \to -\infty$ limit, since we actually need precise decay rates in
order to get an estimate of form
\begin{equation*}
	|G_\lambda(x,y)| \leq C e^{-\eta|x-y|},
\end{equation*}
which are unavailable from $W_j^+$ in the $y \to -\infty$ regime.
Thus, we need to express the $(+)$-bases in terms of the growing
modes $\psi_j^-$ at $-\infty$, and the sole decaying mode $\phi_3^-$
where $\psi_j^-,\phi_3^-$ are defined as in Lemma \ref{lem-estmodes}. 
Expressing such solutions in the basis for $y <
0$, away from zero, there exist {\it analytic} coefficients
$\alpha_{jk}:=\alpha_{jk}(\lambda)$ such that
\begin{equation}\label{eq:goodmodes} 
	\begin{aligned} 
		W_1^+(x,\lambda)&=
			\alpha_{11}(\lambda) \psi_1^-(x,\lambda) + \alpha_{12}(\lambda)
			\psi_2^- (x,\lambda)+ \alpha_{13}(\lambda) \phi_3^-(x,\lambda), \\
		W_2^+ (x,\lambda)&= \alpha_{21} (\lambda)\psi_1^-(x,\lambda) +
			\alpha_{22} (\lambda)\psi_2^- (x,\lambda)
			+ \alpha_{23}(\lambda)\phi_3^-(x,\lambda).
	\end{aligned}
\end{equation}
At $\lambda=0$, we choose $W_1^+(\cdot,0)\equiv \phi_3^-(\cdot,0)\equiv \bar W'$.
Thus, 
\begin{equation}\label{alpha-cond1}
	\alpha_{11}(0) =\alpha_{12}(0) =0.
\end{equation} 
Furthermore, still as $\lambda=0$, $\psi_2^-$ is a (nearly constant) bounded 
solution and has the form $(b_-^{-1},0,1)^\top$ as $x$ near zero. 
Observe also that $W_2^+$ is the solution converging to zero in form of 
$|x|^\nu(1,a(x),a(x))^\top$ as $x\to 0^-$. 
Thus, we can choose
\begin{equation}\label{alpha-cond2}
	\alpha_{22}(0) =0.
\end{equation}
To express the coefficients $C^+_j$ in terms of the uniformly decaying/growing 
modes at $-\infty$, with a slight abuse of notation we write
\begin{equation*}
	\psi_j^- = (u_j^-, q_j^-, p_j^-)^\top, \quad j=1,2,
\end{equation*}
and define the $2 \times 2$ minors
\begin{equation*}
 	\Omega_{ij}^\pm(y,\lambda) 
	:= \begin{vmatrix} q_i^\pm & q_j^\pm \\ p_i^\pm & p_j^\pm \end{vmatrix} 
	= - \Omega_{ji}^\pm (y,\lambda),
\end{equation*}
and the analytic minors
\begin{equation*}
 \hat d_{12}(\lambda) := 
 	\begin{vmatrix} \alpha_{11} & \alpha_{12} \\ \alpha_{21} & \alpha_{22} \end{vmatrix}, \qquad
 \hat d_{23}(\lambda) := 
 	\begin{vmatrix} \alpha_{12} & \alpha_{13} \\ \alpha_{22} & \alpha_{23} \end{vmatrix}, \qquad
 \hat d_{13}(\lambda) := 
 	\begin{vmatrix} \alpha_{11} & \alpha_{13} \\ \alpha_{21} & \alpha_{23} \end{vmatrix}
\end{equation*} 
where by \eqref{alpha-cond1} and \eqref{alpha-cond2}, we note
\begin{equation}\label{dhat-cond}
	 \hat d_{12}(0) = \hat d_{23}(0) =0.
\end{equation}
By elementary column operations we notice that
\begin{equation}\label{minor23}
 	\begin{vmatrix} q_2^+ & q_3^- \\ p_2^+ & p_3^- \end{vmatrix} 
	= \alpha_{21} \Omega_{13}^- + \alpha_{22} \Omega_{23}^-,
\end{equation}
\begin{equation}\label{minor13}
	\begin{vmatrix} q_3^- & q_1^+ \\ p_3^- & p_1^+ \end{vmatrix} 
	= - \alpha_{11} \Omega_{13}^- - \alpha_{12} \Omega_{23}^-,
\end{equation}
\begin{equation}\label{minor12}
 	\begin{vmatrix} q_1^+ & q_2^+\\ p_1^+ & p_2^+ \end{vmatrix} 
	= \Omega_{12}^+ = \hat d_{12} \Omega_{12}^- + \hat d_{13} \Omega_{13}^- 
	+ \hat d_{23} \Omega_{23}^-.
\end{equation}

\begin{lemma}\label{lem:orlambda} 
The minor in \eqref{minor13} can be improved by
\begin{equation}\label{useful3}
	\begin{vmatrix} q_3^- & q_1^+ \\ p_3^- & p_1^+ \end{vmatrix} 
	= \lambda (\hat\alpha_{11} \Omega_{31}^- + \hat\alpha_{12} \Omega_{32}^-)
\end{equation} 
for some coefficients $\hat \alpha_{ij}$.
\end{lemma}
\medskip

\begin{proof} 
The estimate is clear, due to the fact that at $\lambda=0$, we can take 
$W_1^+(x,0) = W_3^-(x,0)= \bar W'(x)$.
\end{proof}
\medskip

We also have the following crucial cancelation for $x>y$
\begin{equation}\label{useful1}
 \begin{aligned}
   W^+_1 C_{11}^+ &+  W_2^+ C_{21}^+ \\
  	&= a^{-1}  D^{-1}_- \Bigl( (\alpha_{21}\Omega_{13}^- + \alpha_{22} \Omega_{23}^-)
  		(\alpha_{11} \psi_1^- + \alpha_{12} \psi^-_2 + \alpha_{13}\phi^-_3)  \\
	&\qquad\qquad\quad
		  - (\alpha_{11}\Omega_{13}^- + \alpha_{12} \Omega_{23}^-)(\alpha_{21} \psi_1^- 
		+\alpha_{22} \psi^-_2 + \alpha_{23}\phi^-_3) \Bigr)\\
	&= a^{-1} D^{-1}_- \left( \hat d_{12} \Omega_{23}^- \psi_1^- - \hat
		d_{12} \Omega_{13}^- \psi_2^- - (\hat d_{13} \Omega_{13}^- 
		+ \hat d_{23} \Omega_{23}^- ) \phi^-_3 \right),
 \end{aligned}
\end{equation}
where $\Omega_{ij}^-$ are functions in $y$ and $\phi_1^-,\phi_2^-, \psi_3^-$ are in $x$, 
noting that  $\Omega_{13}^-\psi_1^-$ and $\Omega_{23}^-\psi_2^-$ are canceled out.

We recall here that $\mu_j^-$, $j=1,2,3$, are three eigenvalues satisfying 
\begin{equation*}
 \mu_1^-\le-c_0<0, \quad
 \mu_2^- = -\lambda/a_- + \cO(\lambda^2),\quad
 \mu_3^-\ge c_0>0,
\end{equation*}
for some $c_0>0$.
\medskip

\begin{lemma}\label{Ccoeffs-est}
For $y<0$ away from zero, we have
\begin{align}
	C_1^+(y,\lambda) &=\lambda^{-1}[u]^{-1}e^{-\mu_2^-\,y}(1,\;-L,\;0) +
	\cO(e^{-\mu_1^-y}+e^{-\mu_2^-y}),\label{estbigy-C1}\\
	C_2^+(y,\lambda)&=\cO(e^{-\mu_1^-y}+e^{-\mu_2^-y})\label{estbigy-C2}\\
	C_3^-(y,\lambda) &=- \lambda^{-1}[u]^{-1}e^{-\mu_2^-\,y} (1,\;-L,\;0)
	+\cO(e^{-\mu_3^-y}).\label{estbigy-C3}
\end{align}
\end{lemma}

\begin{proof} 
First, by using \eqref{eq:goodmodes} and the estimates in previous sections 
on normal modes $\psi_j^-,\phi_3^-$, we readily obtain
\begin{equation*}
	|D_-(y,\lambda)| = |\det (W_1^+,W_2^+,W_3^-)| 
	=\cO(\lambda)e^{\mu_1^-y}e^{\mu_2^-y}e^{\mu_3^-y}
\end{equation*}
and
\begin{equation*}
	|\Omega_{12}^-|=\cO(e^{\mu_1^-y}e^{\mu_2^-y}),\quad
	|\Omega_{13}^-|=\cO(e^{\mu_1^-y}e^{\mu_3^-y}),\quad
	|\Omega_{23}^-|=\cO(e^{\mu_2^-y}e^{\mu_3^-y}).
\end{equation*}
Using \eqref{minor23} and noting that $\alpha_{22}(0)=0$, we estimate
\begin{equation*}
	C_{11}^+(y,\lambda)= -a^{-1}D_-^{-1}
	\Big(\alpha_{21} \Omega_{13}^- + \alpha_{22}\Omega_{23}^-\Big)
	 = \lambda^{-1}\cO(e^{-\mu_2^-y}) + \cO(e^{-\mu_2^-y}+e^{-\mu_1^-y}),
\end{equation*}
which gives \eqref{estbigy-C1} for $\lambda$ small, by observing that the coefficient 
in Laurent  expansions at order $\lambda^{-1}$ is $[u]^{-1}(1,\;-L,\;0)$ 
(see the proof of Lemma \ref{lem-estCnear0}).

Next, by \eqref{useful3}, we estimate
\begin{equation*}
	C_{21}^+(y,\lambda) = a^{-1}D_-^{-1}\,\lambda (\hat\alpha_{11} \Omega_{31}^- 
	+ \hat\alpha_{12}\Omega_{32}^-) = \cO(e^{-\mu_1^-\,y}+e^{-\mu_2^-\,y}).
\end{equation*}
Finally, we can estimate

\begin{equation*}
	\begin{aligned}
		C_{31}^-(y,\lambda) &= a^{-1}D^{-1}_-\Big(\hat d_{12}
			\Omega_{12}^- + \hat d_{13} \Omega_{13}^- + \hat d_{23}
 			\Omega_{23}^-\Big)\\
		&= \lambda^{-1}\cO(e^{-\mu_2^-y}) +
			\cO(e^{-\mu_1^-y}+e^{-\mu_2^-y}+e^{-\mu_3^-y}),
	\end{aligned}
\end{equation*}
in the same way as done for $C_{11}^+$, yielding the estimate as claimed; note that 
the constraint \eqref{dhat-cond} on $\hat d(0)$ shows that only slow-growing mode 
$\psi_2^-$ appears in the $\lambda^{-1}$ term. 
The appearance of the fast-growing term is due to $D^{-1}_-\Omega_{12}^-$, a new 
feature as compared to the estimate of $C_{11}^+$.
\end{proof}
\medskip

\begin{proposition}[Resolvent kernel bounds as $|y|\to +\infty$]\label{prop-awayzero} 
Under \eqref{A0} - \eqref{A5k}, for $|y|$ large, there hold
\begin{equation}\label{G1-est1} 
	\begin{aligned}
	G_\lambda(x,y)=&\lambda^{-1}[u]^{-1}e^{-\mu_2^-\,y}\bar W'(1,\;-L,\;0)\\
	&+\cO((e^{-\mu_2^-\,y}+e^{-\mu_1^-\,y})e^{\mu_3^+\,x})
	\end{aligned}
\end{equation} 
for $y<0<x$, and
\begin{equation}\label{G1-est2} 
	\begin{aligned}
	G_\lambda(x,y) =& \lambda^{-1}[u]^{-1}e^{-\mu_2^-\,y}\bar W'
	(1,\;-L,\;0)\\&+\cO(e^{\mu_1^-(x-y)}) +\cO(e^{\mu_2^-(x-y)}) 
	+\cO(e^{-\mu_2^-y}e^{\mu_3^-x})
	\end{aligned}
\end{equation} 
for $y<x<0$, and
\begin{equation}\label{G1-est3}
	\begin{aligned}G_\lambda(x,y) =&-\lambda^{-1}[u]^{-1}
		e^{-\mu_2^-\,y}\bar W' (1,\;-L,\;0)\\
		&+\cO(e^{-\mu_2^-\,y}e^{\mu_3^-\,x}) +\cO(e^{\mu_3^-\,(x-y)})
	\end{aligned}
\end{equation} 
for $x<y<0$.
Similar bounds can be obtained for the case $y>0$.
\end{proposition}
\medskip

\begin{proof} 
For the first case $y<0<x$, since $W_2^+\equiv 0$ on $(0,+\infty)$, we have
\begin{equation*}
	\begin{aligned}
		\cG_\lambda(x,y) &= W_1^+(x)\, C_1^+(y,\lambda) 
			=\Big( \bar W'(x) +\cO(\lambda e^{\mu_3^+\,x})\Big) 
			C_1^+(y,\lambda) \\
		&=\Big( \bar W'(x) + \cO(\lambda e^{\mu_3^+\,x})\Big)\\
		&\times \Big(\lambda^{-1}[u]^{-1}e^{-\mu_2^-\,y}(1,\;-L,\;0)
			+ \cO(e^{-\mu_1^-\,y}+e^{-\mu_2^-y})\Big),
 	\end{aligned}
\end{equation*} 
giving the estimate \eqref{G1-est1}.

For the third case $x<y<0$, we have
\begin{equation*}
	\begin{aligned}
		\cG_\lambda(x,y) &= -W_3^-(x)\,C_3^-(y,\lambda)  
		=-\Big( \bar W'(x) + \cO(\lambda e^{\mu_3^-\,x})\Big) C_3^-(y,\lambda)\\
		 &=\Big( \bar W'(x) + \cO(\lambda e^{\mu_3^-\,x})\Big)\\
		 &\times \Big(-\lambda^{-1}[u]^{-1}e^{-\mu_2^-\,y}(1,\;-L,\;0) 
		   +\cO(e^{-\mu_1^-y}+e^{-\mu_2^-y}+e^{-\mu_3^-y})\Big),
	\end{aligned}
\end{equation*} 
proving the estimate \eqref{G1-est3}.

Finally, for the second case $y<x<0$, we have
\begin{equation*}
 	\cG_\lambda(x,y) =  W_1^+(x,\lambda)\,C_1^+(y,\lambda)
	+ W_2^+(x,\lambda)\,C_2^+(y,\lambda).
\end{equation*}
In this case, besides the fact that the $\lambda^{-1}$ term comes from 
the expression $W_1^+\,C_1^+$ as above, there is a crucial cancelation 
as computed in \eqref{useful1}, 
which proves \eqref{G1-est2}, using the crucial constraint \eqref{dhat-cond} 
(by our choice  of the bases) to eliminate fast modes in the $\lambda^{-1}$ term.
\end{proof}

\section{Pointwise bounds  and low-frequency estimates} \label{sec:ptwise}

In this section, using the previous pointwise bounds (Propositions \ref{prop-nearzero} 
and \ref{prop-awayzero}) for the resolvent kernel in low-frequency regions, we derive 
pointwise bounds for the ``low-frequency'' Green function:
\begin{equation}\label{LFGreenfn} 
	G^I(x,t;y): = \frac{1}{2\pi i}\int_{\Gamma\bigcap\{|\lambda|\le r\}}
	e^{\lambda t}G_\lambda(x,y)d\lambda
\end{equation}
where $\Gamma$ is any contour near zero, but away from the essential spectrum,
and $r>0$ is a sufficiently small constant such that all previous computations on 
$G_\lambda$ hold.
\medskip

\begin{proposition} \label{prop-greenbounds} 
Assuming \eqref{A0} - \eqref{A5k} and defining the effective diffusion 
$L\,b_\pm$ (see \cite{MaZ3}),  the low-frequency Green distribution $G^I(x,t;y)$ 
associated with the linearized evolution equations  may be decomposed as
\begin{equation*}
	G^I(x,t;y)= E+  \widetilde G^I + R, 
\end{equation*}
where, for $y<0$:
\begin{equation*}
	E(x,t;y):= \bar U_x(x)[u]^{-1} e(y,t),
\end{equation*}
\begin{equation*}
  e(y,t):=\left(\textrm{\rm errfn}\left(\frac{y+a_{-}t}{\sqrt{4\,L\,b_-\,t}}\right)
  -\textrm{\rm errfn}\left(\frac{y-a_{-}t}{\sqrt{4\,L\,b_-\,t}}\right)\right);
\end{equation*}
\begin{equation*}
	|\partial_{x}^\kappa \partial_y^\beta \widetilde G^I(x,t;y)|
	\le C_1\,t^{-(|\beta|+|\kappa|)/2-1/2}e^{-(x-y-a_{-} t)^2/C_2\,t},
\end{equation*}
\begin{equation*}
	R(x,t;y)= \cO(e^{-\eta(|x-y|+t)}) + \cO(e^{-\eta t})\chi(x,y)
	\Big[1+\frac{1}{a(y)}(x/y)^\nu\Big],
\label{Rbounds}
\end{equation*}
for some $\eta$, $C_1$, $C_2>0$, 
where $\beta,\kappa=0, 1$ and $\nu = \frac{Lb(0)+a'(0)}{|a'(0)|}$ and
\begin{equation*}
 	\chi(x,y) =\left\{\begin{aligned} 
 		&1 &\qquad  &-1<y<x<0 \\ 
		&0 &\qquad  & \textrm{\rm otherwise.}
	\end{aligned}\right.
\end{equation*}
Symmetric bounds hold for $y\ge 0$.
\end{proposition}
\medskip

\begin{proof} 
Having resolvent kernel estimates in Propositions \ref{prop-nearzero} and
\ref{prop-awayzero}, we can now follow previous analyses of \cite{ZH,MaZ1,MaZ3}. 
Indeed, the claimed bound for $E$ precisely comes from the term 
$\lambda^{-1}[u]^{-1}e^{-\mu_2^-\,y}\bar U_x$, where 
$\mu_2^-= -\lambda/a_- + \cO(\lambda^2)$.
Likewise, estimates of $\widetilde G^I$ are due to bounds in
Proposition \ref{prop-awayzero} for $y$ away from zero and those in
Proposition \ref{prop-nearzero} for $y$ near zero but $x$ away from zero. 
The singularity occurs only in the case $-1<y<x<0$, as
reported in Proposition \ref{prop-nearzero}. 
In this case, using the estimate \eqref{G0-est2} and moving the contour $\Gamma$ in
\eqref{LFGreenfn} into the stable half-plane $\{\Real\lambda<0\}$, we have
\begin{equation*}
 \int_\Gamma e^{\lambda t}\Big(1+\frac{|x|^{\nu}}{a(y)|y|^\nu}\Big)d\lambda
= \cO(e^{-\eta t})\Big(1+\frac{|x|^{\nu}}{a(y)|y|^\nu}\Big),
\end{equation*}
which precisely contributes to the second term in $R(x,t;y)$. 
The first term is as usual the fast decaying term.
\end{proof}
\medskip

With the above pointwise estimates on the (low-frequency) Green
function, we have the following from \cite{MaZ1,MaZ3}.
\medskip

\begin{lemma}[\cite{MaZ1,MaZ3}]\label{lem-estGI} 
Assuming \eqref{A0} - \eqref{A5k}, $\widetilde G^I$ satisfies
\begin{equation}\label{estGI}
	\Big|\int_{-\infty}^{+\infty} \partial_y^\beta\widetilde G^I(\cdot,t;y) f(y)dy \Big|_{L^p} 
	\le C (1+t)^{-\frac 12 (1/q-1/p)-|\beta|/2}|f|_{L^q},
\end{equation}
for all $t\ge 0$, some $C>0$, for any $1\le q\le p$.
\end{lemma}
\medskip

We recall the following fact from \cite{Z4}.
\medskip

\begin{lemma}[\cite{Z4}]\label{lem-kernel-e} 
The kernel $e$ satisfies
\begin{equation*} 
	\begin{aligned}
	|e_y(\cdot, t)|_{L^p}, |e_t(\cdot, t)|_{L^p}, &\le C t^{-\frac 12(1-1/p)}, \\
	|e_{yt}(\cdot, t)|_{L^p} &\le C t^{-\frac 12 (1-1/p)-1/2}.
	\end{aligned}
\end{equation*}
for all $t> 0$, some $C>0$, for any $p\ge 1$.
\end{lemma}
\medskip

Finally, we have the following estimate on $R$ term.
\medskip

\begin{lemma}\label{lem-estR}
Under \eqref{A0} - \eqref{A5k}, $R(x,t;y)$ satisfies
\begin{equation*} 
	\Big|\int_{-\infty}^{+\infty}R(\cdot,t;y) f(y)dy \Big|_{L^p} \le C
	e^{-\eta t}(|f|_{L^p} + |f|_{L^\infty}),
\end{equation*}
for all $t\ge 0$, some $C,\eta>0$, for any $1\le p\le \infty$.
\end{lemma}
\medskip

\begin{proof} 
The estimate clearly holds for the fast decaying term $e^{-\eta(|x-y|+t)}$ in $R$. 
Whereas, to estimate the second term, first notice that it is only nonzero precisely 
when $-1<y<x<0$ or $0<x<y<1$. Thus, for instance, when $-1<x<0$, we estimate
\begin{equation*}
	\begin{aligned}
		\Big|\int_{-\infty}^{+\infty}\chi(x,y)&\Big[1+\frac{1}{a(y)}(x/y)^\nu\Big]f(y)dy
		\Big|=\Big|\int_{-1}^{x}\Big[1+\frac{1}{a(y)}(x/y)^\nu\Big]f(y)dy
		\Big|\\
		&\le C|f|_{L^\infty}\Big[1+\int_{-1}^{x}\frac{1}{|a(y)|}(x/y)^\nu dy\Big]
		\le C|f|_{L^\infty},
	\end{aligned}
\end{equation*}
where the last integral is bounded by that fact that $a(x) \sim x$ as $|x|\to 0$. 
From this, we easily obtain
\begin{equation*}
	\begin{aligned}
	\Big|\int_{-\infty}^{+\infty}e^{-\eta t}&\chi(x,y)\Big[1+\frac{1}{a(y)}(x/y)^\nu \Big]f(y)dy
	\Big|_{L^p(-1,0)}\le Ce^{-\eta t}|f|_{L^\infty},
	\end{aligned}
\end{equation*}
which proves the lemma.
\end{proof}
\medskip

\begin{remark}\label{rem-singterm}\rm
We note here that the singular term $a^{-1}(y)(x/y)^{\nu}$ appearing in \eqref{G0-est2} 
and \eqref{Rbounds} contributes  to the time-exponential decaying term.
Note that this part agrees with the resolvent kernel for the scalar
convected-damped equation $u_t + au_x=-Lgu,$ for which we can find
explicitly the Green function as a convected time-exponentially decaying
delta function similar to terms appearing in the relaxation or real viscosity case.
\end{remark}

\section{Nonlinear damping estimate}\label{sec:damping}

In this section, we establish an auxiliary damping energy estimate. 
We consider the eventual nonlinear perturbation equations for variables $(u,q)$ 
\begin{equation}\label{eqpert}
	\begin{aligned}
		u_{t}+ (\hat a(u)u)_x +Lq_{x} &=\dot \alpha (U_x+u_x), \\
		-q_{xx} + q +(\hat b(u)\,u)_{x} &=0,
	\end{aligned}
\end{equation}
where $\alpha$ represents the shock location and
\begin{equation*}
	\begin{aligned}
		\hat a(u) := \frac{df}{du}(U+u) = \frac{df}{du}(U) + \cO(|u|) = a(x) + \cO(|u|),\\
		\hat b(u) := \frac{dM}{du}(U+u) = \frac{dM}{du}(U) + \cO(|u|) = b(x) + \cO(|u|).
	\end{aligned}
\end{equation*}
(see subsequent Section \ref{sec:nonlinear}).
In view of
\begin{equation*}
	\hat a_x = a'(x) + \cO(|u|+|u_x|),
\end{equation*}
then, under assumptions \eqref{A4}, \eqref{A5k}, we get that, for all $|u|_\infty$ and 
$|u_x|_\infty$ sufficiently small, there holds
\begin{equation}\label{assump-ag}
	L\hat b + (k+\half)\hat a_x > 0,
\end{equation}
for all $k =1,2,3,4$ and all $x \sim 0$. 
We are going to profit from \eqref{assump-ag} to prove the following
\medskip

\begin{proposition}\label{prop-damping} 
Assume \eqref{A0} - \eqref{A5k}. 
So long as $|u|_{W^{2,\infty}}$ and $|\dot \alpha|$ remain sufficiently small, 
we obtain
\begin{equation}\label{damp-est}
	|u|_{H^k}^2(t)\le Ce^{-\eta t}|u|_{H^{k}}^2(0)+ C\int_0^t
		e^{-\eta(t-s)}(|u|_{L^2}^2+ |\dot\alpha|^2)(s)\,ds, 
		\qquad\eta>0,
\end{equation}
for $k=1,...,4$.
\end{proposition}
\medskip

\begin{proof} 
Let us work for the case $\dot\alpha \equiv 0$. 
The general case will be seen as a straightforward extension. 
For our convenience, we denote the $\phi-$weighted norm as
\begin{equation*}
	|f|_{H^k_\phi}: = \sum_{i=0}^k \iprod{\phi\, \partial_x^if,\partial_x^if}_{L^2}^{1/2}
\end{equation*}
for nonnegative functions $\phi$. 
Now, taking the inner product of $q$ against the second equation in
\eqref{eqpert} and applying the integration by parts, we obtain
\begin{equation*}
	|q_x|_{L^2}^2+|q|_{L^2}^2 
	=\iprod{\hat b u,q_x}\le \half |q_x|^2_{L^2} +C|u|_{L^2}^2.
\end{equation*}
In fact, we also can easily get for $k\geq 1$
\begin{equation}\label{estq}
	|q|_{H^k_\phi}\le C|u|_{H^{k-1}_\phi},
\end{equation} 
Likewise, taking the inner product of $u$ against the first equation
in \eqref{eqpert} and integrating by parts, we get
\begin{equation*}
	\frac 12\dt |u|_{L^2}^2=-\half \int   \hat a_x |u|^2 \, dx-\iprod{Lq_x,u}
\end{equation*} 
which together with \eqref{estq} and the H\"older inequality gives
\begin{equation}\label{u0est}
	\dt |u|_{L^2}^2 \le C |u|_{L^2}^2.
\end{equation}
In order to establish estimates for derivatives, for each $k\ge1$
and $\phi\ge0$ to be determined later, we compute
\begin{equation}\label{k-est}
	\frac12 \dt\iprod{\partial_x^k u,\phi\,\partial_x^k u} 
	=\iprod{\partial_x^k u_t,\phi\,\partial_x^k u}
	=-\iprod{L\partial_x^{k+1}q + \partial_x^{k+1}(\hat au),\phi\,\partial_x^k u}
\end{equation}
where, using the equation for $q$, we estimate
\begin{equation*}
	\begin{aligned}
	-\iprod{L\partial_x^{k+1}q,\phi\,\partial_x^k u} &= -\iprod{L\partial_x^{k}(\hat b u) 
		+L\partial_x^{k-1}q,\phi\,\partial_x^k u}\\
	&\le-\iprod{L\,\hat b\,\phi\,\partial_x^{k}u,\partial_x^k u} + \epsilon |\partial_x^k u|_{L^2_\phi}^2 
		+ C_\epsilon\Big[|u|_{H^{k-1}_\phi}^2+|q|_{H^{k-1}_\phi}^2\Big]\\
	&\le-\frac\eta2|\partial_x^k u|_{L^2_\phi}^2 + C|u|_{H^{k-1}_\phi}^2
	\end{aligned}
\end{equation*}
and
\begin{equation*}
	\begin{aligned}
	-\iprod{\partial_x^{k+1}(\hat au),\phi\,\partial_x^k u}
	&= -\iprod{\hat a\partial_x^{k+1}u+(k+1)\hat a_x\partial_x^{k}u 
	+ L.O.T,\phi\,\partial_x^k u}\\
	&=\iprod{ \Big(\half(\hat a\phi)_x 
	-(k+1)\hat a_x\phi\Big)\partial_x^{k}u ,\partial_x^k u}
	- \iprod{L.O.T,\phi\,\partial_x^k u}\\
	&\le\iprod{ \Big(\half(\hat a\phi)_x -(k+1)\hat a_x\phi\Big)\partial_x^{k}u ,\partial_x^k u} 
	+ \epsilon|\partial_x^ku|_{L^2_\phi}^2+C_\epsilon|u|_{H^{k-1}_\phi}^2
	\end{aligned}
\end{equation*}
By choosing $\phi:=|\hat a|^{2k+1}$, we observe that 
\begin{equation*}
	\half(\hat a\phi)_x-(k+1)\hat a_x\phi\equiv0
\end{equation*} 
and thus
\begin{equation*}
	-\iprod{\partial_x^{k+1}(\hat au),\phi\,\partial_x^k u} \le
	\epsilon|\partial_x^ku|_{L^2_\phi}^2+C_\epsilon|u|_{H^{k-1}_\phi}^2
\end{equation*}
for any positive number $\epsilon$. 
Taking $\epsilon$ small enough and putting these above estimates together 
into \eqref{k-est}, we have just obtained
\begin{equation}\label{wk-est}
	\begin{aligned} 
	\dt\iprod{\partial_x^k u,|\hat a|^{2k+1} \partial_x^k u}
 	&\le -\eta_1\iprod{\partial_x^k u,|\hat a|^{2k+1}\partial_x^k u} + C|u|_{H^{k-1}}^2,
	\end{aligned}
\end{equation} 
for each $k\ge 1$ and some small $\theta_1>0$.

In addition, by choosing $\phi \equiv 1$ in \eqref{k-est}, we obtain
\begin{equation}\label{k-est1}
	\begin{aligned} 
	\frac12\dt\iprod{\partial_x^k u,\partial_x^k u}
 	&= -\iprod{\Big(L\hat b + (k+\half)\hat a_x\Big)\partial_x^k u,\partial_x^k u} +
	\epsilon|\partial_x^ku|_{L^2}^2+C_\epsilon|u|_{H^{k-1}}^2,
	\end{aligned}
\end{equation} 
for any $\epsilon>0$. 
By assumption \eqref{assump-ag}, there exist $\eta_2$ sufficiently small and
$M>0$ sufficiently large such that
\begin{equation}\label{abound} 
	M \theta_1 |\hat a|^{2k+1} + \Big(L\hat b + (k+\half)\hat a_x\Big) \ge \eta_2>0,
\end{equation}
for all $x \in \R$ (by taking $M$ large enough away from zero; for $x \sim 0$ 
the bound follows from \eqref{assump-ag}).
Therefore, by adding \eqref{k-est1} with $M$ times \eqref{wk-est}, using \eqref{abound}, 
and taking $\epsilon = \eta_2/2$ in \eqref{k-est1}, we obtain
\begin{equation}\label{k-semifinal}
 	\dt\iprod{(1+M|\hat a|^{2k+1})\partial_x^k u,\partial_x^k u}
 	\le -\frac{\eta_2}{2}|\partial_x^k u|_{L^2}^2 + C|u|_{H^{k-1}}^2.
\end{equation}
Now, for $\delta>0$, let us define 
\begin{equation*}
 	\cE(t):= \sum_{i=0}^k\delta^i \iprod{(1+M|a|^{2k+1})\partial_x^k u,\partial_x^k u}.
\end{equation*}  
Observe that $\cE(t) \sim |u|_{H^k}^2$. We then use \eqref{u0est} and \eqref{k-semifinal}
for $k=1,...,4$ and take $\delta$ sufficiently small to derive
\begin{equation*}
 	\dt\cE(t) \le -\eta_3\, \cE(t) + C |u|_{L^2}^2(t)
\end{equation*}
for some $\eta_3>0$, from which \eqref{damp-est} follows 
by the standard Gronwall's inequality.
\end{proof}

\section{High--frequency estimate}
\label{sec:hfest}

In this section, we estimate the high--frequency part of the solution 
operator  $e^{\cL t}$ (see \eqref{iLT})
\begin{equation}\label{formS2}
	\cS_2(t) =\frac{1}{2\pi i}\int_{-\gamma_1-i\infty}^{-\gamma_1+i\infty}
	\chi_{{}_{\{|\Imag\lambda|\geq\gamma_2\}}}e^{\lambda t} (\lambda-\cL)^{-1} d\lambda,
\end{equation}
for small constants $\gamma_1,\gamma_2>0$
(here $\chi_{{}_{I}}$ is the characteristic function of the set $I$).
\medskip

\begin{proposition}[High-frequency estimate] \label{prop-HFest} 
Under assumptions \eqref{A0} - \eqref{A5k}, we obtain
\begin{equation}\label{HFsoln-est}
		|\partial^\kappa_x\cS_2(t)(\phi-L\,\partial_x(\cK \psi))|_{L^{2}}\le  C
			e^{-\eta_1t}\Big(|\psi|_{H^{\kappa+2}}+|\varphi|_{H^{\kappa+2}}\Big)
			\qquad\kappa=0,1,
\end{equation}
for some $\eta_1>0$, where $\cK = (-\partial_x^2 + 1)^{-1}$ and $L$ is
a constant (see \eqref{eq:redlin}).
\end{proposition}
\medskip

Our first step in proving \eqref{HFsoln-est} is to estimate the solution 
of the resolvent system
\begin{equation*}
\begin{aligned}
    \lambda u + (a(x)\,u)_{x} +L q_{x} &=\varphi,\\
    - q_{xx} +q +(b(x)\, u)_{x} &=\psi,
\end{aligned}
\end{equation*} 
where $a(x)=\dfrac{df}{du}(U(x))$ and $b(x)=\dfrac{dM}{du}(U)$ as before.
\medskip

\begin{proposition}[High-frequency bounds]\label{prop-resHF} 
Assuming \eqref{A0} - \eqref{A5k}, for some $R,C$ sufficiently large
and $\gamma>0$ sufficiently small, we obtain
\begin{equation*}
\begin{aligned}
	|(\lambda - \cL)^{-1}(\varphi-L\partial_x (\cK \psi))|_{H^1} 
	&\le C \Big( |\varphi|_{H^1}^2+|\psi|_{L^2}^2 \Big),\\
	|(\lambda - \cL)^{-1}(\varphi-L\partial_x (\cK \psi))|_{L^2} 
	&\le \frac{C}{|\lambda|^{1/2}}\Big(|\varphi|_{H^1}^2+|\psi|_{L^2}^2\Big),
\end{aligned}
\end{equation*}
for all $|\lambda|\ge R$ and $\Real\lambda \ge -\gamma$.
\end{proposition}
\medskip

\begin{proof} 
A Laplace transformed version of the nonlinear energy estimates 
\eqref{u0est} and \eqref{k-semifinal} in Section \ref{sec:damping} with $k = 1$
(see \cite{Z7}, pp. 272--273, proof of Proposition 4.7 for further details) yields
\begin{equation}\label{Re-est} 
	\begin{aligned} 
		\Big(\Real\lambda+\frac{\gamma_1}{2}\Big)|u|_{H^1}^2\le C\Big(|u|_{L^2}^2 
		+ |\varphi|_{H^1}^2+|\psi|_{L^2}^2\Big).
	\end{aligned}
\end{equation}
On the other hand, taking the imaginary part of the $L^2$ inner product of $U$ 
against $\lambda u = \cL u + \partial_xL\cK h + f$ and applying the Young's inequality, 
we also obtain the standard estimate
\begin{equation}\label{Im-est}
	\begin{aligned} 
		|\Imag\lambda||u|_{L^2}^2&\le |\iprod{\cL u,u}| 
		+ |\iprod{L\cK \psi,u_x}| + |\iprod{\varphi,u}| \\
		&\le C \Big(|u|_{H^1}^2 + |\psi|_{L^2}^2 + |\varphi|_{L^2}^2\Big),
	\end{aligned}
\end{equation}
noting the fact that $\cL$ is a bounded operator from $H^1$ to $L^2$
and $\cK$ is bounded from $L^2$ to $H^1$.

Therefore, taking $\gamma=\gamma_1/4$, we obtain from \eqref{Re-est} and \eqref{Im-est}
\begin{equation*}
	|\lambda||u|_{L^2}^2 + |u|_{H^1}^2\le C \Big(|u|_{L^2}^2 
	+ |\psi|_{L^2}^2 + |\varphi|_{H^1}^2\Big),
\end{equation*}
for any $\Real\lambda \ge -\gamma$. 
Now take $R$ sufficiently large such that $|u|_{L^2}^2$ on the right  hand side 
of the above can be absorbed into the left hand side for $|\lambda|\ge R$, thus yielding
\begin{equation*}
	|\lambda||u|_{L^2}^2 + |u|_{H^1}^2
	\le C \Big( |\psi|_{L^2}^2 + |\varphi|_{H^1}^2\Big),
\end{equation*}
for some large $C>0$, which gives the result as claimed.
\end{proof}

Next, we have the following
\medskip

\begin{proposition}[Mid-frequency bounds]\label{prop-resMF}
Assuming \eqref{A0} - \eqref{A5k}, we obtain
\begin{equation*}
	|(\lambda - \cL)^{-1}\varphi|_{L^2} \le C\,|\varphi|_{H^1}
	\quad \textrm{ for } \;
	R^{-1}\le |\lambda|\le R \mbox{ and }\Real\lambda \ge -\gamma,
\end{equation*} 
for any $R$ and $C=C(R)$ sufficiently large and $\gamma = \gamma(R)>0$ 
sufficiently small.
\end{proposition}
\medskip

\begin{proof}
Immediate, by compactness of the set of frequency under consideration together 
with the fact that the resolvent $(\lambda-\cL)^{-1}$ is analytic with respect to $H^{1}$ 
in $\lambda$; see, for instance, \cite{Z4}. 
\end{proof}
\medskip

With Propositions \ref{prop-resHF} and \ref{prop-resMF} in hand, we are now ready to give:
\medskip

\begin{proof}[Proof of Proposition \ref{prop-HFest}] 
The proof starts with the following resolvent identity, using analyticity on the 
resolvent set $\rho(\cL)$ of the resolvent $(\lambda-\cL)^{-1}$, for all  $\varphi\in \mathcal{D}(\cL)$,
\begin{equation*}
(\lambda-\cL)^{-1}\varphi=\lambda^{-1}(\lambda-\cL)^{-1}\cL \varphi+\lambda^{-1}\varphi.
\end{equation*}
Using this identity and \eqref{formS2}, we estimate
\begin{equation*}
	\begin{aligned}
		\cS_2(t)\varphi &=\frac{1}{2\pi i}\int_{-\gamma_1-i\infty}^{-\gamma_1+i\infty}
		\chi_{{}_{\{|\Imag\lambda|\geq\gamma_2\}}}e^{\lambda t}
		\lambda^{-1}(\lambda-\cL)^{-1}\cL\,\varphi\,d\lambda\\
		&\quad+\frac{1}{2\pi i}\int_{-\gamma_1-i\infty}^{-\gamma_1+i\infty}
		\chi_{{}_{\{|\Imag\lambda|\geq\gamma_2\}}} e^{\lambda t}\lambda^{-1}\varphi\,d\lambda\\
		&=:S_1 + S_2,
	\end{aligned}
\end{equation*}
where, by Propositions \ref{prop-HFest} and \ref{prop-resMF}, we have
\begin{equation*}
	\begin{aligned}
		|S_1|_{L^2}&\le C \int_{-\gamma_1-i\infty}^{-\gamma_1+i\infty}
		|\lambda|^{-1}e^{\Real \lambda t}|(\lambda-\cL)^{-1}\cL
		\varphi|_{L^2}|d\lambda|\\
		&\le C e^{-\gamma_1 t}\int_{-\gamma_1-i\infty}^{-\gamma_1+i\infty} |\lambda|^{-3/2}|\cL
		\varphi|_{H^1}|d\lambda|\\
		&\le C e^{-\gamma_1t}|\varphi|_{H^{2}}
	\end{aligned}
\end{equation*}
and
\begin{equation*}
	\begin{aligned}
	|S_2|_{L^2}&\leq\frac{1}{2\pi }\Big|\varphi\int_{-\gamma_1-i\infty}^{-\gamma_1+i\infty}
	\lambda^{-1}e^{\lambda t} d\lambda\Big|_{L^2}+ \frac{1}{2\pi }
	\Big|\varphi\int_{-\gamma_1-i r}^{-\gamma_1+i r}\lambda^{-1}e^{\lambda t} d\lambda \Big|_{L^2}\\
	&\leq Ce^{-\gamma_1 t} |\varphi|_{L^2},
\end{aligned}
\end{equation*}
by direct computations, noting that the integral in $\lambda$ in the first term is identically zero.
This completes the proof of the bound for the term involving $\varphi$ as stated in the proposition. 
The estimate involving $\psi$ follows by observing that $L\,\partial_x \cK$ is bounded from $H^s$ to $H^s$.
Derivative bounds can be obtained similarly. 
\end{proof}
\medskip

\begin{remark}\label{rem-HF}\rm
We note that in our treating the high-frequency terms by energy estimates (as also done 
in \cite{KZ,NZ2}), we are ignoring the pointwise contribution there, which would also be 
convected time-decaying delta functions. 
To see these features, a simple exercise is to do the Fourier transform of the equations 
about a constant state.
\end{remark}

\section{Nonlinear analysis}\label{sec:nonlinear}

In this section, we shall prove the main nonlinear stability
theorem. Following \cite{HZ,MaZ3}, define the nonlinear perturbation
\begin{equation}\label{per-variable}
	\begin{pmatrix}u\\q \end{pmatrix}(x,t): =
	\begin{pmatrix}\tilde u\\\tilde q \end{pmatrix}(x+\alpha(t),t) -
	\begin{pmatrix}U\\Q \end{pmatrix}(x),
\end{equation}
where the shock location $\alpha(t)$ is to be determined later.

Plugging \eqref{per-variable} into \eqref{modelsyst}, we obtain the perturbation equation
 \begin{equation*}
    \begin{aligned}
    u_{t}+ (a(x)\,u)_x  +L\, q_{x}&=N_1(u)_x + \dot\alpha(t)\,(u_x+U_x),\\
     - q_{xx} +  q +(b(x)\,u)_{x} &=N_2(u)_x,
    \end{aligned}
\end{equation*}
where $N_j(u)=O(|u|^2)$ so long as $u$ stays uniformly bounded.

We decompose the Green function as
\begin{equation}\label{Greendecomp}
	G(x,t;y) = G^I(x,t;y) +G^{II}(x,t;y)
\end{equation} where
$G^I(x,t;y)$ is the low-frequency part. 
We further define as in Proposition \ref{prop-greenbounds},
\begin{equation*}
	\widetilde G^I(x,t;y) = G^I(x,t;y) - E(x,t;y)  - R(x,t;y)
\end{equation*}
and
\begin{equation*}
	\widetilde G^{II}(x,t;y) = G^{II}(x,t;y) + R(x,t;y).
\end{equation*}
Then, we immediately obtain, from Lemmas
\ref{lem-estGI}, \ref{lem-estR} and Proposition \ref{prop-HFest}, the following
\medskip

\begin{lemma}\label{lem-estGI+II} 
There holds
\begin{equation}\label{est-tGI}
	\Big|\int_{-\infty}^{+\infty} \partial_y^\beta\widetilde
	G^I(\cdot,t;y) f(y)dy \Big|_{L^p} \le C (1+t)^{-\frac 12
	(1/q-1/p)-|\beta|/2}|f|_{L^q},
\end{equation}
for all $1\le q\le p, \beta=0,1,$ and
\begin{equation}\label{est-tGII}
	\Big|\int_{-\infty}^{+\infty}\widetilde G^{II}(x,t;y)f(y)dy
	\Big|_{L^p} \le C e^{-\eta t}|f|_{H^3},
\end{equation} 
for all $2\le p\le \infty$.
\end{lemma}
\medskip

\begin{proof} 
Bound \eqref{est-tGI} is precisely the estimate \eqref{estGI} in Lemma \ref{lem-estGI}, 
recalled here for our convenience. 
Inequality \eqref{est-tGII} is a straightforward combination of Lemma \ref{lem-estR} 
and Proposition \ref{prop-HFest}, followed by a use of the interpolation inequality 
between $L^2$ and $L^\infty$ and an application of the standard Sobolev imbedding. 
\end{proof}
\medskip

We next show that by Duhamel's principle we have:
\medskip

\begin{lemma} 
There hold the reduced integral representations:
\begin{equation}\label{int-rep}
	\begin{aligned}
		u(x,t)=& \int_{-\infty}^{+\infty} (\widetilde G^I+\widetilde G^{II})(x,t;y)u_0(y)dy \\
		&- \int_0^t \int_{-\infty}^{+\infty} \widetilde G_y^I(x,t-s;y)
			\Big(N_1(u) - L\,\cK\,\partial_y N_2(u) + \dot\alpha\,u\Big)(y,s)\, dy\, ds\\
		&+ \int_0^t \int_{-\infty}^{+\infty}\widetilde G^{II}(x,t-s;y)
			\Big( N_1(u) - L\,\cK\,\partial_y N_2(u) + \dot\alpha\,u\Big)_y(y,s)\, dy\, ds,\\
		q(x,t) =& (\cK \partial_x)( N_2(u)-b\,u) (x,t),
\end{aligned}
\end{equation}
and
\begin{equation}\label{alpha-rep}
	\begin{aligned}
		\alpha(t)=& -\int_{-\infty}^{+\infty}e(y,t)u_0(y)dy \\
		&+ \int_0^t \int_{-\infty}^{+\infty} e_{y}(y,t-s)
			\Big(N_1(u) - L\,\cK\,\partial_y N_2(u) + \dot\alpha\,u\Big)(y,s)\, dy\, ds.
	\end{aligned}
\end{equation}
\begin{equation}\label{alphader-rep}
	\begin{aligned}
		\dot\alpha(t)=& -\int_{-\infty}^{+\infty}e_t(y,t)u_0(y)dy \\
		&+ \int_0^t \int_{-\infty}^{+\infty}e_{yt}(y,t-s)
			\Big(N_1(u) - L\,\cK\,\partial_y N_2(u) + \dot\alpha\,u\Big)(y,s)\, dy\, ds.
	\end{aligned}
\end{equation}
\end{lemma}
\medskip

\begin{proof} 
By Duhamel's principle and the fact that
\begin{equation*}
	\int_{-\infty}^{+\infty}G(x,t;y) U'(y)dy = e^{\cL t} U'(x) = U'(x),
\end{equation*}
we obtain
\begin{equation*}
	\begin{aligned}
		u(x,t)=& \int_{-\infty}^{+\infty} G(x,t;y)u_0(y)dy \\
		&+ \int_0^t \int_{-\infty}^{+\infty}G(x,t-s;y)
			\Big(N_1(u) - L\,\cK\,\partial_y  N_2(u) + \dot\alpha\,u\Big)_y(y,s)\, dy\, ds\\
		&+\alpha(t)\,U'.
	\end{aligned}
\end{equation*}
Thus, by defining the {\it instantaneous shock location:}
\begin{equation*}
	\begin{aligned}
		\alpha(t)=& -\int_{-\infty}^{+\infty}e(y,t)u_0(y)dy \\
		&+ \int_0^t \int_{-\infty}^{+\infty}e_{y}(y,t-s)
			\Big(N_1(u) - L\,\cK\,\partial_y  N_2(u) + \dot\alpha\,u\Big)(y,s)\, dy\, ds
	\end{aligned}
\end{equation*}
and using the Green function decomposition \eqref{Greendecomp}, we
easily obtain the integral representation as claimed in the lemma.
\end{proof}
\medskip

With these preparations, we are now ready to prove the main theorem,
following the standard stability analysis of \cite{MaZ4,Z3,Z4}:
\medskip

\begin{proof}[Proof of Theorem \ref{theo-main}]
Define 
\begin{equation*}
	\zeta(t):=\sup_{0\le s\le t, 2\le p\le \infty} 
	\Big[|u(s)|_{L^p}(1+s)^{\frac 12(1-1/p)}
	+|\alpha(s)|+|\dot\alpha(s)|(1+s)^{1/2}\Big].
\end{equation*}
We shall prove here that for all $t\ge 0$ for which a solution exists with $\zeta(t)$ 
uniformly bounded by some fixed, sufficiently small constant, there holds
\begin{equation}\label{zeta-est}
	\zeta(t) \le C(|u_0|_{L^1\cap H^s}+\zeta(t)^2) .
\end{equation}
This bound together with continuity of $\zeta(t)$ implies that
\begin{equation}\label{zeta-est1} 
	\zeta(t) \le 2C|u_0|_{L^1\cap H^s}
\end{equation}
for $t\ge0$, provided that $|u_0|_{L^1\cap H^s}< 1/4C^2$. 
This would complete the proof of the bounds as claimed in the theorem, 
and thus give the main theorem.

By standard short-time theory/local well-posedness in $H^s$, and the standard principle 
of continuation, there exists a solution $u\in H^s$ on the open time-interval for which 
$|u|_{H^s}$ remains bounded, and on this interval $\zeta(t)$ is well-defined and continuous. 
Now, let $[0,T)$ be the maximal interval on which $|u|_{H^s}$ remains strictly bounded by 
some fixed, sufficiently small constant $\delta>0$. 
By Proposition \ref{prop-damping}, and the Sobolev embedding inequality 
$|u|_{W^{2,\infty}}\le C|u|_{H^s}$, $s\ge3$, we have
\begin{equation}\label{Hs}
	\begin{aligned}
		|u(t)|_{H^s}^2 &\le Ce^{-\eta t}|u_0|_{H^s}^2
			+ C \int_0^t e^{-\eta(t-\tau)}\Big(|u(\tau)|_{L^2}^2
			+|\dot\alpha(\tau)|^2 \Big)d\tau\\
		&\le C(|u_0|_{H^s}^2+\zeta(t)^2)(1+t)^{-1/2}
	\end{aligned}
\end{equation}
and so the solution continues so long as $\zeta$ remains small, with
bound \eqref{zeta-est1}, yielding existence and the claimed bounds.

Thus, it remains to prove the claim \eqref{zeta-est}. 
First by representation \eqref{int-rep} for $u$, for any $2\le p\le \infty$,
we obtain
\begin{equation*}
	\begin{aligned}
		|u|_{L^p} (t) &\le \Big|\int_{-\infty}^{+\infty} 
			(\widetilde G^I+\widetilde G^{II})(x,t;y)u_0(y)dy \Big|_{L^p}\\
		&+ \int_0^t\Big| \int_{-\infty}^{+\infty} \widetilde G_y^I(x,t-s;y)
			\Big(N_1(u) - L\,\cK\,\partial_y K N_2(u) + \dot\alpha\,u\Big)(y,s)\, dy\Big|_{L^p} ds\\
		&+ \int_0^t \Big|\int_{-\infty}^{+\infty}\widetilde G^{II}(x,t-s;y)
			\Big(N_1(u) - L\,\cK\,\partial_y  N_2(u) +\dot\alpha\,u\Big)_y(y,s)\, dy\Big|_{L^p} ds\\
		&= I_1 + I_2 + I_3,
	\end{aligned}
\end{equation*}
where estimates \eqref{est-tGI} and \eqref{est-tGII} yield
\begin{equation*}
	\begin{aligned}
		I_1&= \Big|\int_{-\infty}^{+\infty} (\widetilde G^I+\widetilde
			G^{II})(x,t;y)u_0(y)dy \Big|_{L^p} \\
		&\le C(1+t)^{-\frac 12(1-1/p)}|u_0|_{L^1} + Ce^{-\eta t}|u_0|_{H^3} \\
		&\le C(1+t)^{-\frac 12(1-1/p)}|u_0|_{L^1\cap H^3},
	\end{aligned}
\end{equation*}
and, with noting that $L\, \cK\,\partial_y $ is bounded from $L^2$ to $L^2$,
\begin{equation*}
\begin{aligned} 
	I_2 &= \int_0^t\Big| \int_{-\infty}^{+\infty}
		\widetilde G_y^I(x,t-s;y) \Big(N_1(u) - L\,\cK\,\partial_y N_2(u)  
		+\dot\alpha\,u\Big)(y,s)\, dy\Big|_{L^p} ds\\
	&\le C \int_0^t(t-s)^{-\frac 12 (1/2-1/p)-1/2}(|u|_{L^\infty} + |\dot\alpha|)|u|_{L^2}(s)ds\\
	&\le C \zeta(t)^2\int_0^t(t-s)^{-\frac 12 (1/2-1/p)-1/2}(1+s)^{-3/4}ds\\
	&\le C \zeta(t)^2(1+t)^{-\frac 12 (1-1/p)},
\end{aligned}
\end{equation*}
and, together with \eqref{Hs}, $s\ge 4$,
\begin{equation*}
	\begin{aligned}
	I_3&= \int_0^t \Big|\int_{-\infty}^{+\infty}\widetilde
		G^{II}(x,t-s;y) \Big(N_1(u) - L\,\cK\,\partial_y  N_2(u)  
		+\dot\alpha\,u\Big)_y(y,s)\, dy\Big|_{L^p} ds\\
	&\le C\int_0^t e^{-\eta (t-s)} |N_1(u) - L\,\cK\,\partial_y  N_2(u) 
		+\dot\alpha\,u|_{H^4}(s)ds\\
	&\le C\int_0^t e^{-\eta (t-s)}
		(|u|_{H^s} + |\dot\alpha|)|u|_{H^s}(s)ds\\&\le C(|u_0|_{H^s}^2
		+\zeta(t)^2)\int_0^t e^{-\eta (t-s)} (1+s)^{-1}ds\\
	&\le C(|u_0|_{H^s}^2 +\zeta(t)^2)(1+t)^{-1}.
	\end{aligned}
\end{equation*}
Thus, we have proved
\begin{equation*}
	|u(t)|_{L^p}(1+t)^{\frac 12(1-1/p)}\le C(|u_0|_{L^1\cap H^s} +\zeta(t)^2).
\end{equation*}
Similarly, using representations \eqref{alpha-rep} and \eqref{alphader-rep} and the 
estimates in Lemma \ref{lem-kernel-e} on the kernel $e(y,t)$, we can estimate 
(see, e.g., \cite{MaZ4,Z4}),
\begin{equation*}
	|\dot\alpha(t)|(1+t)^{1/2} + |\alpha(t)|\le C(|u_0|_{L^1} +\zeta(t)^2).
\end{equation*}
This completes the proof of the claim \eqref{zeta-est}, and thus the
result for $u$ as claimed. 
To prove the result for $q$, we observe that $\cK\,\partial_x$ is bounded 
from $L^p\to W^{1,p}$ for all $1\le p\le \infty$, and thus from the representation 
\eqref{int-rep} for $q$, we estimate
\begin{equation*}
	\begin{aligned} 
		|q|_{W^{1,p}}(t)&\le C(|N_2(u)|_{L^p}+ |u|_{L^p})(t) \\
		&\le C|u|_{L^p}(t)\le C|u_0|_{L^1\cap H^s}(1+t)^{-\frac12(1-1/p)}
	\end{aligned}
\end{equation*} 
and
\begin{equation*}
		|q|_{H^{s+1}}(t)\le C|u|_{H^s}(t)\le C|u_0|_{L^1\cap H^s}(1+t)^{-1/4},
\end{equation*} 
which complete the proof of the main theorem. 
\end{proof}

\appendix
\section{Pointwise reduction lemma}\label{firstAppendix}
Let us consider the situation of a system of equations of form
\begin{equation} \label{eq:firstorder} 
	W_x = \A^\epsilon(x,\lambda)W,
\end{equation}
for which the coefficient $\A^\epsilon$ does not exhibit uniform
exponential decay to its asymptotic limits, but instead is {\it
slowly varying} (uniformly on a $\epsilon$-neighborhood $\cV$, being
$\epsilon>0$ a parameter). 
This case occurs in different contexts for rescaled equations, such 
as \eqref{eq:strechtedsyst} in the present analysis.

In this situation, it frequently occurs that not only $\A^\epsilon$
but also certain of its invariant eigenspaces are slowly varying
with $x$, i.e., there exist matrices
\begin{equation*}
	{\mathbb{L}}^\epsilon=\begin{pmatrix} L^\epsilon_1 \\
	L^\epsilon_2\end{pmatrix}(x), \quad \R^\epsilon=\begin{pmatrix}
	R^\epsilon_1 & R^\epsilon_2\end{pmatrix} (x) \label{LR}
\end{equation*}
for which ${\mathbb{L}}^\epsilon \R^\epsilon(x)\equiv I$ and
$|{\mathbb{L}}\R'|=|{\mathbb{L}}'\R|\le C\delta^\epsilon(x)$, uniformly in $\epsilon$,
where the pointwise error bound $\delta^\epsilon =
\delta^\epsilon(x)$ is small, relative to
\begin{equation}\label{M}
	{\mathbb{M}}^\epsilon:= {\mathbb{L}}^\epsilon \A^\epsilon \R^\epsilon(x)
	=\begin{pmatrix} M^\epsilon_1 & 0 \\
		0 & M^\epsilon_2 \end{pmatrix}(x)
\end{equation}
and ``$'$'' as usual denotes $\partial/\partial x$.  
In this case, making the change of coordinates $W^\epsilon=\R^\epsilon Z$, we 
may reduce \eqref{eq:firstorder} to the approximately block-diagonal equation
\begin{equation}\label{eq:blockdiag}
	{Z^\epsilon}'= {\mathbb{M}}^\epsilon Z^\epsilon + \delta^\epsilon
	\Theta^\epsilon Z^\epsilon, 
\end{equation}
where ${\mathbb{M}}^\epsilon$ is as in \eqref{M}, $\Theta^\epsilon(x)$ is a
uniformly bounded matrix, and $\delta^\epsilon(x)$ is (relatively) small. 
Assume that such a procedure has been successfully carried
out, and, moreover, that there exists an approximate {\it uniform
spectral gap in numerical range}, in the strong sense that
\begin{equation*} \label{eq:gap}
	\min \sigma(\Real M_1^\epsilon)- \max \sigma(\Real M_2^\epsilon) 
	\ge \eta^\epsilon(x), \qquad \textrm{\rm for all } x,
\end{equation*}
with pointwise gap $\eta^\epsilon(x) > \eta_0 > 0$ 
uniformly bounded in $x$ and in $\epsilon$; here and elsewhere 
$\Real N:= \half (N+N^*)$ denotes the ``real'', or symmetric part of an operator $N$. 
Then, there holds the following {\it pointwise reduction lemma}, a refinement of 
the reduction lemma of \cite{MaZ3} (see the related ``tracking lemma'' given in 
varying degrees of generality in \cite{GZ,MaZ1,PZ,ZH,Z3}).
\medskip

\begin{proposition}\label{pwrl} 
Consider a system \eqref{eq:blockdiag} under the gap
assumption \eqref{eq:gap}, with $\Theta^\epsilon$ uniformly bounded
in $\epsilon \in \cV$ and for all $x$.  If, for all $\epsilon \in
\cV$, $\sup_{x \in \R} (\delta^\epsilon/\eta^\epsilon)$ is
sufficiently small (i.e., the ratio of pointwise gap
$\eta^\epsilon(x)$ and pointwise error bound $\delta^\epsilon(x)$ is
uniformly small), then there exist (unique) linear transformations
$\Phi_1^\epsilon(x,\lambda)$ and $\Phi_2^\epsilon(x,\lambda)$,
possessing the same regularity with respect to the various
parameters $\epsilon$, $x$, $\lambda$ as do coefficients
${\mathbb{M}}^\epsilon$ and $\delta^\epsilon(x)\Theta^\epsilon(x)$, for which
the graphs $\{(Z_1, \Phi^\epsilon_2 (Z_1))\}$ and
$\{(\Phi^\epsilon_1(Z_2),Z_2)\}$ are invariant under the flow of
\eqref{eq:blockdiag}, and satisfying
\begin{equation*}
	\sup_\R |\Phi^\epsilon_j| \, \le \, C \sup_\R (\delta^\epsilon/\eta^\epsilon).
\end{equation*}
Moreover, we have the pointwise bounds
\begin{equation}\label{ptwise}
	|\Phi^\epsilon_2(x)|\le C \int_{-\infty}^x e^{-\int_y^x
	\eta^\epsilon(z)dz} \delta^\epsilon(y) dy,
\end{equation}
and symmetrically for $\Phi^\epsilon_1$.
\end{proposition}
\medskip

\begin{proof}
By a change of independent coordinates, we may arrange that $\eta^\epsilon(x)\equiv$ 
constant, whereupon the first assertion reduces to the conclusion of the tracking/reduction 
lemma of \cite{MaZ3}.  
Recall that this conclusion was obtained by seeking $\Phi^\epsilon_2$ as the solution 
of a fixed-point equation
\begin{equation*}
	\Phi^\epsilon_2(x)= {\mathcal T }\Phi^\epsilon_2(x):= \int_{-\infty}^x
	\cF^{y\to x} \delta^\epsilon(y) Q(\Phi^\epsilon_2)(y) dy.
\end{equation*}
Observe that in the present context we have allowed $\delta^\epsilon$ to vary with $x$, 
but otherwise follow the proof of \cite{MaZ3} word for word to obtain the conclusion 
(see Appendix C of \cite{MaZ3}, proof of Proposition 3.9). 
Here, $Q(\Phi^\epsilon_2)=\cO(1+|\Phi^\epsilon_2|^2)$ by construction, and
$|\cF^{y\to x}|\le Ce^{-\eta(x-y)}$. 
Thus, using only the fact that $|\Phi^\epsilon_2|$ is bounded, we obtain the 
bound \eqref{ptwise} as claimed, in the new coordinates for which $\eta^\epsilon$ 
is constant.  
Switching back to the old coordinates, we have instead 
$|\cF^{y\to x}|\le Ce^{-\int_y^x \eta^\epsilon(z)dz}$, yielding the result in the general case.
\end{proof}
\medskip

\begin{remark}\label{rem:reduced}\rm
From Proposition \ref{pwrl}, we obtain reduced flows
\begin{equation*}
	\left\{\begin{aligned}
		{Z_1^\epsilon}' &= M_1^\epsilon Z_1^\epsilon + \delta^\epsilon(
			\Theta_{11} + \Theta_{12}^\epsilon \Phi_2^\epsilon) Z_1^\epsilon,\\
		{Z_2^\epsilon}' &= M_2^\epsilon Z_2^\epsilon +
			\delta^\epsilon(\Theta_{22}+ \Theta_{21}^\epsilon \Phi_1^\epsilon)
			Z_2^\epsilon.
	\end{aligned}\right.
\end{equation*}
on the two invariant manifolds described.
\end{remark}

\section{Spectral stability}\label{appx-condD} 

Consider the eigenvalue system \eqref{spectralsyst}. 
Integrating the equations we find the zero-mass conditions for $u$ and $q$,
\begin{equation*}
	\int_\R u \, dx = 0, \qquad \int_\R q \, dx = 0,
\end{equation*}
which allows us to recast system \eqref{spectralsyst} in terms of the integrated 
coordinates, which we denote, again, as $u$ and $q$. 
The result is
\begin{equation}\label{evalueeq}
    \begin{aligned}
    \lambda u+ a(x)\,u'  +L q'&=0,\\ - q''+  q +b(x)\,u' &=0.
    \end{aligned}
\end{equation}
The following proposition is the main result of this section.
\medskip

\begin{proposition}\label{prop-spectralcond} 
Let $(u,q)$ be a bounded solution of \eqref{evalueeq}, corresponding to a 
complex number $\lambda\ne 0$.
Then $\Real\lambda <0$ provided that at least one of the following conditions holds\\
(i) $b$ is a constant;\\
(ii) $|u_+ - u_-|$ is sufficiently small. 
\end{proposition}
\medskip

\begin{proof} 
In any case, we can assume $b>0$ by redefining $q$ by $-q$ if necessary, 
still preserving the condition $Lb>0$.
Taking the real part of the inner product of the first equation
against $b\,\bar u$ and using integration by parts, we obtain
\begin{equation*}
	\begin{aligned}
	\Real\lambda |b^{1/2}u|_{L^2}^2 &= -\Real \iprod{a\,b\,u',u} - \Real\iprod{Lq',gu} \\
	&= \Real \Big(\iprod{(a\,b)' u,u} + \iprod{Lq,(b\,u)'}\Big)\\
	&= \Real \Big(\iprod{(a\,b)' u,u} + \iprod{Lq,q''-q + b'\,u}\Big)\\
	&= \Real\Big( \iprod{a'\,b\,u,u} - \iprod{L\,q',q'}-\iprod{Lq,q} 
		+\iprod{a\,b'\,u,u}+\iprod{L\,b'\,q,u}\Big)\\
	&\le \iprod{a'\,b\,u,u} -\frac L2|q|_{H^1}^2+ C\iprod{(|a|+|b'|)|b'|u,u},
 \end{aligned}
\end{equation*}
which proves the proposition in the first case, noting  $a' =\dfrac{d^2 f}{du^2}(U)\,U'<0$ 
(by monotonicity of the profile)  and $b\ge \theta>0$. 
For the second case, observe that $|a|+|b'|$ is now sufficiently small and $|b'|$ and $|a'|$ 
have the same order of ``smallness'', that is, of order $\cO(|U'|) = \cO(|u_+-u_-|)$. 
Thus, the last term on the right-hand side of the above estimate can be absorbed into 
the first term, yielding the result for this second case as well.
\end{proof}

\section{Monotonicity of profiles under nonlinear coupling}
\label{app:exmon}

In this Appendix we show that radiative scalar shock profiles for general nonlinear coupling are monotone, a feature which plays a key role in our stability analysis. Although the existence of profiles for nonlinear coupling is already addressed in \cite{LMS2}, and the monotonicity for the linear coupling case is discussed in \cite{Ser7,LMS2}, for completeness (and convenience of the reader) we closely review the (scalar) existence proof of \cite{LMS1} and extend it to the nonlinear coupling case, a procedure which leads to monotonicity in a very simple way.

The main observation of this section is precisely that, thanks to assumptions  \eqref{A0} and \eqref{A4}, the mapping $u \mapsto LM(u)$ is a diffeomorphism on its range \cite{LMS2}, which can be regarded as the identity along the arguments of the proof leading to the existence result of \cite{LMS1}. 
Since $LM$ is monotone increasing in $[u_+,u_-]$, setting $M_\pm:=M(u_\pm)$,
there exists an inverse function  $H : [LM_+,LM_-] \to [u_+,u_-]$ such that
\begin{equation*}
	y = LM(u) \iff u = H(y),
\end{equation*}
for each $u \in [u_+,u_-]$ and with derivative
\begin{equation*}
	\frac{d H}{dy} = \left(L\,\frac{dM}{du}(H(y))\right)^{-1} \, > \, 0.
\end{equation*}

Consider once again the stationary profile equations \eqref{eq:profileeqn}
(after appropriate flux normalizations), with $(U,Q)(\pm\infty) = (u_\pm,0)$. 
Integration of the equation for $Q$ leads to $\int_\R Q = -[M]=M_--M_+$. 
Let us introduce the variable $Z$ as
\begin{equation*}
	Z:= -L\int_{-\infty}^x Q(\xi) \, d\xi + Lb_-,
\end{equation*}
such that $Z' = -LQ$ and $Z \to LM_\pm$ as $x \to +\infty$. 
In terms of the new variable $Z$ the profile equations are
\begin{align*}
	Z'' &= f(U)' ,\\ Z'- Z''' &= LM(u)'.
\end{align*}
Integrating las equations, and using the asymptotic limits for $Z$, we arrive at the system
\begin{equation}\label{newsystforU}
	\begin{aligned}
		Z' &= f(U) - f(u_\pm),\\
		Z - Z'' &= LM(u).
	\end{aligned}
\end{equation}
We can thus rewrite the ODE for $Z$ as
\begin{equation}\label{our23}
	Z' = F(H(Z-Z'')),
\end{equation}
where $F(u) := f(u) - f(u_\pm)$. 
In view of strict convexity of $f$, the function $F$ is strictly decreasing in the interval 
$[u_+,u_*]$ and strictly increasing in $[u_*,u_-]$, with $F(u_\pm) = 0$ and $F(u_*) = -m < 0$. 
Hence, $F$ is invertible in those intervals with corresponding inverse functions $h_\pm$, 
and we look at the solutions to two ODEs, namely,
\begin{equation}\label{newtwoODEs}
	\begin{aligned}
		Z'' = Z - LM(h_\pm(Z')),\\
		Z(\pm\infty) = L\,M_\pm, \quad Z'(\pm\infty) = 0,
	\end{aligned}
\end{equation}
in their corresponding intervals of existence. 
Observe that the derivatives of the functions $h_\pm$ are given by $h_\pm' = 1/f'(h_\pm(\cdot))$, 
with $f'(u) \neq 0$ in $[u_+,u_*) \cup (u_*,u_-]$. 
Note that $h_+ : [-m,0] \to [u_+,u_*]$ and $h_- : [-m,0] \to [u_*,u_-]$, and that $h_+$ ($h_-$) 
is monotonically decreasing (increasing) on its domain of definition.

Following \cite{LMS1} closely, we shall exhibit the existence of a $Z$-profile solution to \eqref{our23} 
between the states $LM_- > LM_+$, for which the velocity profile follows by $U = H(Z-Z'')$ 
(see \eqref{newsystforU}). 
In the sequel we only indicate the differences with the proofs in Section 2 of \cite{LMS1}, 
and pay particular attention to the monotonicity properties of $Z$, which leads to the monotonicity 
of $U$ in Lemma \ref{monoU} below.

The following proposition is an extension of Propositions 2.2 and 2.3 
in \cite{LMS1} to the variable $G'$ case.
\medskip

\begin{proposition}[\cite{LMS1}]\label{propour23}
(i) Denote $Z_+ = Z_+(x)$ the (unique up to translations) maximal solution to
\begin{equation*}
	Z'' = Z - LM(h_+(Z')),
\end{equation*}
with conditions $Z(+\infty) = Lb_+$ and $Z'(+\infty) = 0$. 
Then $Z_+$ is monotone increasing, ${Z'_+}$ is monotone decreasing, and $Z_+$ 
is not globally defined, that is, there exists a point that we can take without loss of 
generality as $x=0$ (because of translation invariance) such that
\begin{equation*}
	{Z_+}(0) - {Z''_+}(0) = LM(u_*),\qquad
	{Z'_+}(0) = -m < 0.
\end{equation*}
(ii) Denote $Z_- = Z_-(x)$ the (unique up to translations) maximal solution to
\begin{equation*}
	Z'' = Z - LM(h_-(Z')),
\end{equation*}
with conditions $Z(-\infty) = Lb_-$ and $Z'(-\infty) = 0$. 
Then $Z_-$ and $Z'_-$ are monotone increasing, and $Z_-$ is not globally defined, 
that is, there exists a point that we can take without loss of generality as $x=0$ 
(because of translation invariance) such that
\begin{equation*}
	{Z_-}(0) - {Z''_-}(0) = LM(u_*),\qquad
	{Z'_-}(0) = -m < 0.
\end{equation*}
\end{proposition}
\medskip

\begin{proof}
We focus on part (i) of the Proposition.  The second part is analogous. 
Rewrite the equation for $Z_+$ as $X' = J_+(X)$ with $X = (Z,Z')^\top$ and
\begin{equation*}
	J_+ (X) = \begin{pmatrix}Z' \\ Z - LM(h_+(Z')) \end{pmatrix},
\end{equation*}
for which
\begin{equation*}
	{\nabla J_+}_{|(LM_+, 0)} 
	= \begin{pmatrix} 0 & 1 \\ 1 & \dfrac{-L}{f'(u_+)}\,\dfrac{dM}{du}(u_+)
	\end{pmatrix},
\end{equation*}
in view of $h_+(0) = u_+$, and therefore, the starting point $(LM_+,0)$ of the trajectory is a saddle point. 
We focus on the stable manifold as we need $Z$ to be decreasing. 
Follow the trajectory that exits from $(LM_+,0)$ in the lower half plane of the phase field $(Z,Z')$. 
We claim that $Z$ is strictly monotone decreasing and $Z'$ is strictly monotone increasing. 
Suppose, by contradiction, that $Z$ attains a local maximum at $x_0 \in \R$. 
Then $Z'(x_0) = 0$ and $0 \geq Z''(x_0) = (Z - LM(h_+(Z'))_{x=x_0} = Z(x_0) - LM_+$, which is false. 
Hence, $Z$ is monotone decreasing and $Z' < 0$. 
Now, assume that $Z'$ attains a local minimum at $x=x_0$. 
Then the trajectory $Z' = \varphi(Z)$ in the phase plane must attain a local minimum at the same point, 
yielding $\varphi'(Z) = 0$ and $\varphi''(Z) \geq 0$. 
Thus, at $x = x_0$,
\begin{equation*}
	0=\varphi'(Z) = Z'' / Z' = (Z - LM(h_+(Z')))/Z'
\end{equation*}
and
\begin{equation*}
	\begin{aligned}
		\varphi''(Z) &= (d/dZ)((Z - LM(h_+(Z')))/Z') \\
		&= 1/Z' - (dZ'/dZ) \left( (Z' + LM'(h_+(Z'))h'_+(Z') Z' - LM(h_+(Z')))/(Z')^2\right).
	\end{aligned}
\end{equation*}
But $(dZ'/dZ) = \varphi'(Z) = 0$ at $x = x_0$, thus $\varphi''(Z) = 1/Z' < 0$, 
which is a contradiction. 
This shows that $Z'$ is strictly monotone increasing with $Z'' > 0$, and clearly 
$LM(h_+(Z')) \in [LM_+, LM(u_*)]$, $h_+(Z') \in [u_+,u_*]$.
This shows that $Z'' = Z + \mathcal{O}(1)$ and the solution does not blow up in finite time.

By following the proof of Proposition 2.2 in \cite{LMS1} word by word from this point on, 
it is possible to show that the solution reaches the boundary of definition of the differential 
equation at a finite point  which, by translation invariance, we can take as $x = 0$. 
Hence, ${Z'_+}(0) - {Z''_+}(0) = LM(u_*)$ and ${Z'_+}(0) = -m < 0$ hold. 
This concludes the proof.
\end{proof}
\medskip

\begin{lemma}\label{lemour24}
For the maximal solutions $Z_\pm$ of Proposition \ref{propour23}, there holds
\begin{equation*}
	Z_-(0) \leq LM(u_*) \leq Z_+(0).
\end{equation*}
\end{lemma}
\medskip

\begin{proof}
This follows by mimicking the proof of Lemma 2.4 in \cite{LMS1}. 
We warn the reader to now consider the dynamical system
\begin{equation*}
	\begin{aligned}
		y' &= F(H(y)),\\
		y(\pm\infty) &= LM_\pm.
	\end{aligned}
\end{equation*}
A comparison of the solution $y$ of the system above with the trajectory $Z_+$ 
in the phase space yields the inequality on the right. 
The other inequality is analogous. See \cite{LMS1} for details.
\end{proof}
\medskip

The last lemma guarantees the existence of a point of intersection for the orbits of the 
maximal solutions $Z_+$ and $Z_-$ in the phase state field. 
The monotonicity of $Z_\pm$ and $Z'_\pm$ implies that the intersection is unique. 
Matching  the two trajectories at that point provides the desired $Z$-profile. 
Hence, we have the following extension of the existence result in \cite{LMS1} (Theorem 2.5).
\medskip

\begin{theorem}[\cite{LMS1}]\label{existencethm}
Under assumptions, there exists a (unique up to translations) $Z$-profile of class $C^1$ 
with $Z(\pm\infty) = LM_\pm$, solution to \eqref{our23}. 
The solution $Z$ is of class $C^2$ away from a single point, where $Z'$ has at most a jump discontinuity.
Moreover, there exists a (unique up to translations) velocity profile $U$ with $U(\pm\infty) = u_\pm$ 
solution to \eqref{newsystforU}, which is continuous away from a single point, where it has at most 
a jump discontinuity satisfying Rankine-Hugoniot conditions and the entropy condition.
\end{theorem}
\medskip

\begin{proof}
Lemma \ref{lemour24}  implies the existence of a point in the $(Z,Z')$ plane where the graphs 
of $Z_-$ and $Z_+$ intersect. 
By monotonicity of the graphs the intersection is unique. 
Thus, after  an appropriate translation,  we can find a point $\bar x \in \R$ such that 
$(Z_-(\bar x), Z'_-(\bar x))  = (Z_+(\bar x), Z'_+(\bar x)) =: (\hat Z, \hat Y)$, and the 
$Z$-profile is defined as
\begin{equation*}
	Z(x) := \begin{cases} Z_+(x), & x \geq \bar x,\\ Z_-(x), & x \leq \bar x.\end{cases}
\end{equation*}
$Z$ is $C^1$ and satisfies $Z \to Lb_\pm$ as $x \to \pm \infty$. 
Moreover, $Z$ is $C^2$ except at $x = \bar x$. 
The velocity profile is now defined via
\begin{equation*}
 	U := H(Z-Z''),
\end{equation*}
with the described regularity properties due to regularity of $Z$ and the fact that 
$H = (LM)^{-1}$ is of class, at least, $C^2$. 
Likewise, at the only possible discontinuity $x=\bar x$ of $U$ is is possible to prove that $U$ 
satisfies Rankine-Hugoniot condition, $U(\bar x -0) = U(\bar x + 0)$ and the entropy 
condition $U(\bar x -0) = h_-(\hat Y) > h_+(\hat Y) = U(\bar x + 0)$.
\end{proof}
\medskip

\begin{lemma}[Monotonicity]\label{monoU}
The constructed profile $U$ is strictly monotone decreasing.
\end{lemma}
\medskip

\begin{proof}
Let $x_2 > x_1$, with $x_i \neq \bar x$, and suppose that $U(x_2) \geq U(x_1)$, that is, 
$H(Z-Z'')_{|x=x_2} \geq H(Z-Z'')_{|x=x_1}$. Since $H$ is strictly monotone increasing we 
readily have that
\begin{equation*}
	LM(h_\pm(Z'_\pm(x_2)) = (Z- Z'')_{|x=x_2} \geq (Z-Z'')_{x=x_1}
	= LM(h_\pm(Z'_\pm(x_1)),
\end{equation*}
where the $\pm$ sign depends on which side of $x= \bar x$ we are evaluating the $Z$-profile. 
Suppose $x_1$, $x_2$ are on the same side, say, $\bar x < x_1 < x_2$ 
(the symmetric case, $x_1 < x_2 < \bar x$, is analogous). 
Since $LM$ is monotonically increasing, last condition implies that 
$h_+(Z'_+(x_2)) \geq h_+(Z'_+(x_1))$. 
But this is a contradiction with the fact that $Z_+'$ is monotone increasing and $h_+$ is 
strictly decreasing, yielding $h_+(Z'_+(x_2)) < h_+(Z'_+(x_1))$. 
The case $x_1 < \bar x < x_2$ leads to the condition $h_+(Z'_+(x_2)) \geq h_-(Z'_-(x_1))$, 
which is obviously false in view that $h_+ : [-m,0] \to [u_+,u_*]$ and $h_- : [-m,0] \to [u_*,u_-]$, 
yielding again a contradiction. 
Finally, we remark that at the only point of discontinuity of $U$, namely at $x=\bar x$, 
the jump is entropic, satisfying $U(\bar x - 0) > U(\bar x + 0)$. 
Therefore $U$ is strictly monotone decreasing in all $x \in \R$.
\end{proof}
\medskip

\begin{remark}\label{regulrem}\rm
Observe that the constructed velocity profile is continuous, except, at most, at one point 
where it observes an entropic jump. 
The regularity of $U$ increases as long as the strength of the profile decreases below an 
explicit threshold \cite{KN1,LMS1}, becoming continuous and, moreover, of class $C^2$. 
We remark, however, that away from the possible discontinuity $x=\bar x$, the profile has 
the same regularity of $Z''$, independently of the shock strength, because of smoothness of $H$. 
Whence, from regularity assumption \eqref{A0} and by differentiating equation \eqref{newtwoODEs}, 
$Z''$ is of class $C^2$ away from $x = \bar x$, and so is $U$. 
Finally, thanks to translation invariance we have chosen $x=0$ to be the point where the 
equations for the profiles $Z_\pm$ reach $LM(u_*)$, being $u_*$ the only zero of $\frac{df}{du}(u)$; 
this implies that $U(0) = H(Z-Z'')_{x=0} = H(LM(h_\pm(-m))) = u_*$, so that $a(x) = \frac{df}{du}(U)$ 
vanishes only at $x=0$.
\end{remark}

In view of last remark we have the following
\medskip

\begin{corollary}\label{Ureg}
Except for a possible single point $x = \bar x$, the profile $U$ is of class $C^2$ 
and satisfies $U' < 0$ a.e. 
Moreover, the function $a(x) :=\frac{df}{du}(U)$ is of class $C^1$ except at a 
point $x=\bar x$, and  vanishes only at $x = 0$ (by translation invariance).
\end{corollary}
\medskip

Finally, we state the regularity properties for the convex flux, based on the 
analysis in \cite{LMS1}, Section 3. The proof is, once again, an adaptation to the 
general $G$ case of the proof of Proposition 3.3 in \cite{LMS1},  which we omit.
\medskip

\begin{corollary}
Under convexity of the velocity flux, $\frac{d^2f}{du^2}> 0$, if the shock amplitude $|u_+ - u_-|$ 
is sufficiently small then the profile is of class $C^2$ and $U'(x) < 0$ for all $x \in \R$.
\end{corollary}

\end{document}